\documentclass[a4paper,12pt, abstract=true]{scrartcl}

\usepackage{german}
\usepackage[ansinew]{inputenc}
\usepackage[german,english,french,USenglish]{babel}
\usepackage{latexsym,amssymb,amsfonts,bbm,amsmath} 
\usepackage{theorem}
\usepackage{xcolor}
\usepackage{graphicx}
\usepackage{esint}
\usepackage{stmaryrd}
\usepackage[normalem]{ulem} 
\newtheorem{dummytheorem}{Dummy-Theorem}[section]
\newcommand{\proofendsign}{$\Box$} 
\newtheorem{definition}[dummytheorem]{Definition}
\newtheorem{lemma}[dummytheorem]{Lemma}
\newtheorem{theorem}[dummytheorem]{Theorem}
\newtheorem{proposition}[dummytheorem]{Proposition}

\newtheorem{corollary}[dummytheorem]{Corollary}

\newenvironment{proof}{{\noindent \bf Proof }}
 {{\hspace*{\fill}\proofendsign\par\bigskip}}

\theorembodyfont{\normalfont}
\newtheorem{remarknorm}[dummytheorem]{Remark}
\newtheorem{examplenorm}[dummytheorem]{Example}

\newcommand{\V}{\mathbf{V}}

\newcommand{\bE}{\mathbf{E}}
\newcommand{\bD}{\mathbf{D}}
\newcommand{\bC}{\mathbf{C}}

\newcommand{\N}{\mathbb{N}}

\newcommand{\Q}{\mathbb{Q}}
\newcommand{\R}{\mathbb{R}}

\newcommand{\F}{\mathbb{F}}

\newcommand{\G}{\mathbb{G}}

\newcommand{\pr}{\mathbb{P}}
\newcommand{\ex}{\mathbb{E}}
\newcommand{\vari}{\mathbb{V}{\rm ar}}
\newcommand{\covi}{\mathbb{C}{\rm ov}}

\newcommand{\eins}{\mathbbm{1}}

\newcommand{\cadlag}{c\`adl\`ag}

\usepackage[comma]{natbib}

\allowdisplaybreaks

\begin{document}


\title{Functional delta-method for the bootstrap of quasi-Hadamard differentiable functionals}

\author{
Eric Beutner\footnote{Department of Quantitative Economics, Maastricht Univ., {\tt e.beutner@maastrichtuniversity.nl}}
\qquad\qquad
Henryk Zähle\footnote{Department of Mathematics, Saarland University, {\tt zaehle@math.uni-sb.de}}}
\date{} 
\maketitle

\begin{abstract}
The functional delta-method provides a convenient tool for deriving the asymptotic distribution of a plug-in estimator of a statistical functional from the asymptotic distribution of the respective empirical process. Moreover, it provides a tool to derive bootstrap consistency for plug-in estimators from bootstrap consistency of empirical processes. It has recently been shown that the range of applications of the functional delta-method for the asymptotic distribution can be considerably enlarged by employing the notion of quasi-Hadamard differentiability. Here we show in a general setting that this enlargement carries over to the bootstrap. That is, for quasi-Hadamard differentiable functionals bootstrap consistency of the plug-in estimator follows from bootstrap consistency of the respective empirical process. This enlargement often requires convergence in distribution of the bootstrapped empirical process w.r.t.\ a nonuniform sup-norm. The latter is not problematic as will be illustrated by means of examples.
\end{abstract}

{\bf Keywords:} Bootstrap; Functional delta-method; Quasi-Hadamard differentiability; Statistical functional; Weak convergence for the open-ball $\sigma$-algebra.

\bigskip
\bigskip



\newpage

\section{Introduction}\label{section introduction}

The bootstrap is a widely used technique to approximate the unknown error distribution of estimators. Since the seminal paper by \citet{Efron1979} many variants of his bootstrap
procedure have been introduced in the literature. Furthermore, the bootstrap has quickly been extended to other data than a sample of independent and identically distributed random variables. For general accounts on the bootstrap one may refer to \citet{Efron1994}, \citet{ShaoTu1995}, \citet{Davison and Hinkley(1997)}, \citet{Lahiri2003}, among others.

For a (tangentially) Hadamard differentiable map $f$ the functional delta-method leads to the asymptotic distribution of $a_n(f(\widehat T_n)-f(\theta))$ whenever the asymptotic distribution of $a_n(\widehat T_n-\theta)$ is known. Here $\widehat T_n$ is an estimator for a (possibly infinite-dimensional) parameter $\theta$, and $(a_n)$ is a sequence of real numbers tending to infinity such that $a_n(\widehat T_n-\theta)$ has a non-degenerate limiting distribution. This extends to the bootstrap, i.e.\ bootstrap consistency of $a_n(f(\widehat T_n^*)-f(\widehat T_n))$ follows from bootstrap consistency of $a_n(\widehat T_n^*-\widehat T_n)$ for (tangentially) Hadamard differentiable $f$; see, for instance, \citet[Theorems 3.9.11 and 3.9.13]{van der Vaart Wellner 1996}. Here $\widehat T_n^*$ is a bootstrapped version of $\widehat T_n$ based on some random mechanism. For a recent partial generalization of these results, see also \citet{VolgushevShao2014}. \cite{Parr1985} established a functional delta-method for the bootstrap of Fréchet differentiable maps $f$, and \cite{CuevasRomo1997} obtained a corresponding result for the so-called smoothed bootstrap.

A drawback of the classical functional delta-method is its restricted range of applications. For many statistical functionals $f$ (including classical L-, V- and M-functionals) the condition of (tangential) Hadamard differentiability is simply too strong. For this reason \citet{BeutnerZaehle2010} introduced the notion of quasi-Hadamard differentiability, which is weaker than (tangential) Hadamard differentiability but still strong enough to obtain a generalized version of the classical functional delta-method; see also the Appendix \ref{appendix Functional Delta-Method}. Combined with results for weak convergence of empirical processes w.r.t.\ nonuniform sup-norms the concept of quasi-Hadamard differentiability led to some new weak convergence results for plug-in estimators of statistical functionals based on dependent data; see \citet{BeutnerZaehle2010,BeutnerZaehle2012}, \cite{Ahn and Shyamalkumar2011}, \citet{BeutnerWuZaehle2012}, \cite{Kraetschmeretal2013}, and \cite{KraetschmerZaehlel2016}. See also \citet{BeutnerZaehle2013} and \citet{Buchsteiner2015} for some recent results on weak convergence of empirical processes w.r.t.\ nonuniform sup-norms.

In this article, we will show that the notion of quasi-Hadamard differentiability admits even a functional delta-method for the bootstrap. This enlarges the set of functionals $f$ for which bootstrap consistency of $a_n(f(\widehat T_n^*)-f(\widehat T_n))$ follows immediately from bootstrap consistency of $a_n(\widehat T_n^*-\widehat T_n)$. To illustrate this, let us briefly discuss distortion risk functionals as examples for $f$ where the parameter $\theta$ is a distribution function $F$ on the real line, $\widehat T_n$ represents the empirical distribution function $\widehat F_n$ of $n$ real-valued random variables with distribution function $F$, and $\widehat T_n^*$ corresponds to a bootstrapped version $\widehat F_n^*$ of $\widehat F_n$.

Given a continuous concave distortion function $g$, i.e.\ a concave function $g:[0,1]\rightarrow[0,1]$ being continuous at $0$ and satisfying $g(0)=0=1-g(1)$, the corresponding {\em distortion risk functional} $f_g:\bD(f_g)\rightarrow\R$ is defined by
\begin{equation}\label{def dist risk meas}
    f_g(F)\,:=\, \int_{-\infty}^0 g(F(t))\,dt-\int_0^\infty \big(1-g(F(t))\big)\,dt,
\end{equation}
where $\bD(f_g)$ is a suitable subset of the set of all distribution functions $F$ for which both integrals on the right-hand side are finite. Note that distortion risk functionals associated with continuous concave distortion functions correspond to coherent distortion risk measures (cf.\ Example \ref{Example DR Functional}) which are of special interest in mathematical finance and actuarial mathematics. It was discussed in \cite{BeutnerZaehle2010} and \cite{Kraetschmeretal2013} that these functionals are typically  not Hadamard differentiable w.r.t.\ the usual sup-norm $\|\cdot\|_\infty$ but only quasi-Hadamard differentiable w.r.t.\ suitable nonuniform sup-norms $\|v\|_\phi:=\|v\phi\|_\infty$ stronger than $\|\cdot\|_\infty$ (i.e.\ with continuous weight functions $\phi:\R\rightarrow[1,\infty)$ satisfying $\lim_{|x|\to\infty}\phi(x)=\infty$). The functional delta-method in the form of Corollary \ref{modified delta method for the bootstrap - II} below then shows that $a_n(f_g(\widehat F_n^*)-f_g(\widehat F_n))$ has the same limiting distribution as $a_n(f_g(\widehat F_n)-f_g(F))$ whenever the bootstrapped empirical process $a_n(\widehat F_n^*-\widehat F_n)$ converges in distribution to the same limit as the empirical process $a_n(\widehat F_n-F)$. As ``differentiability'' can be obtained only for certain nonuniform sup-norms $\|\cdot\|_\phi$, the latter convergence in distribution has to be guaranteed for exactly these nonuniform sup-norms $\|\cdot\|_\phi$. Fortunately, such results can be easily obtained from Donsker results for appropriate classes of functions; see Sections \ref{bootstrap results for empirical processes - iid}--\ref{bootstrap results for empirical processes - beta mixing} 
for examples. So the notion of quasi-Hadamard differentiability together with the functional delta-method based on it provides an interesting field of applications for the bootstrap of Donsker classes. We emphasize that our approach leads in particular to new bootstrap results for empirical distortion risk measures based on $\beta$-mixing data; for details and other examples see Section \ref{bootstrap results for empirical processes - applications}.

It is worth recalling that the empirical process $a_n(\widehat F_n-F)$, regarded as a mapping from $\Omega$ to the nonseparable space of all bounded \cadlag\ functions equipped with the sup-norm, is not measurable w.r.t.\ the Borel $\sigma$-algebra. This problem was first observed by \citet{Chibisov1965} and carries over to nonuniform sup-norms. There are different ways to deal with this fact; for a respective discussion see, for instance, Section 1.1 in \cite{van der Vaart Wellner 1996}. One possibility is to use the concept of weak convergence (or convergence in distribution) in the Hoffmann-J{\o}rgensen sense; see, for instance, \cite{van der Vaart Wellner 1996}, \citet{Dudley1999}, \cite{Lahiri2003}, and \cite{Kosorok2007}. Another possibility is to use the open-ball $\sigma$-algebra w.r.t.\ which the empirical process is measurable. Here we work throughout with the open-ball $\sigma$-algebra and weak convergence (and convergence in distribution) as defined in \citet[Section 6]{Billingsley1999}; see also \citet{Dudley1966,Dudley1967}, \citet{Pollard1984}, and \cite{ShorackWellner1986}. This implies in particular that we have to take proper care of the measurability of the maps $a_n(\widehat T_n-\theta)$ and $a_n(\widehat T_n^*-\widehat T_n)$ for every $n \in \N$.

The rest of the article is organized as follows. In Section \ref{Auxiliaries} we briefly explain the setting chosen here and give some definitions that will be used throughout. The main result and its proof are presented in Sections \ref{Delta-Method for the bootstrap} and \ref{Proof of Abstract Delta Method}, respectively. Applications of our main result are given in Section \ref{application to statistical functionals} and illustrated in Section \ref{bootstrap results for empirical processes}. Additional definitions and results that are needed for our main result are given in the Appendix. The Appendix is organized as follows. In Sections \ref{sec weak conv for open ball} and \ref{appendix Conv Iin Dist and Prob} we give some results on weak convergence, convergence in distribution, and convergence in probability for the open-ball $\sigma$-algebra which are needed in Section \ref{appendix Functional Delta-Method}. In Section \ref{appendix Functional Delta-Method} we first present an extended Continuous Mapping theorem for convergence in distribution for the open-ball $\sigma$-algebra. This complements the extended Continuous Mapping theorems for weak convergence for the Borel $\sigma$-algebra and for convergence in distribution in the Hoffmann-J{\o}rgensen sense which are already known from the literature. In the second part of Section \ref{appendix Functional Delta-Method} we use the extended Continuous Mapping theorem to prove an extension (compared to Theorem 4.1 in \citet{BeutnerZaehle2010}) of the functional delta-method based on the notion of quasi-Hadamard differentiability. This extension is needed for the proof of our main result, i.e.\ for the proof of a functional delta-method for the bootstrap. Two results that ensure measurability of maps involved in our approach are given in Section \ref{section topology 30}.


\section{Basic definitions}\label{Auxiliaries}

In this section we introduce some notation and basic definitions. As mentioned in the introduction, weak convergence and convergence in distribution will always be considered for the open-ball $\sigma$-algebra. Borrowed from \citet[Section 6]{Billingsley1999} we will use the terminology {\em weak$^\circ$ convergence} (symbolically $\Rightarrow^\circ$) and {\em convergence in distribution$^\circ$} (symbolically $\leadsto^\circ$). For details see the Appendices \ref{sec weak conv for open ball} and \ref{appendix Conv Iin Dist and Prob}. In a separable metric space the notions of weak$^\circ$ convergence and convergence in distribution$^\circ$ boil down to the conventional notions of weak convergence and convergence in distribution for the Borel $\sigma$-algebra. In this case we also use the symbols $\Rightarrow$ and $\leadsto$ instead of $\Rightarrow^\circ$ and $\leadsto^\circ$, respectively.

Let $\V$ be a vector space and $\bE$ be a subspace of $\V$. Let $\|\cdot\|_{\bE}$ be a norm on $\bE$ and ${\cal B}^\circ$ be the corresponding open-ball $\sigma$-algebra on $\bE$.
Let $(\Omega,{\cal F},\pr)$ be a probability space, and $(\widehat T_n)$ be a sequence of maps
$$
    \widehat T_n:\Omega\longrightarrow\V.
$$
Regard $\omega\in\Omega$ as a sample drawn from $\pr$, and $\widehat T_n(\omega)$ as a statistic derived from $\omega$. Let $\theta\in\V$, and $(a_n)$ be a sequence of positive real numbers tending to $\infty$. Assume that $a_n(\widehat T_n-\theta)$ takes values only in $\bE$ and is $({\cal F},{\cal B}^\circ)$-measurable for every $n\in\N$, and that
\begin{equation}\label{abstract bootstrap - eq - 10}
    a_n(\widehat T_n-\theta)\,\leadsto^\circ\,\xi\qquad\mbox{in $(\bE,{\cal B}^\circ,\|\cdot\|_{\bE})$}
\end{equation}
for some $(\bE,{\cal B}^\circ)$-valued random variable $\xi$.

Now, let $(\Omega',{\cal F}',\pr')$ be another probability space and set
$$
    (\overline\Omega,\overline{\cal F},\overline{\pr}):=(\Omega\times\Omega',{\cal F}\otimes{\cal F}',\pr\otimes\pr').
$$
The probability measure $\pr'$ represents a random experiment that is run independently of the random sample mechanism $\pr$. In the sequel, $\widehat T_n$ will frequently be regarded as a map defined on the extension $\overline\Omega$ of $\Omega$. Let
$$
    \widehat T_n^*: \overline\Omega\longrightarrow\V
$$
be any map and assume that $a_n(\widehat T_n^*-\widehat T_n)$ takes values only in $\bE$ and is $(\overline{\cal F},{\cal B}^\circ)$-measurable for every $n\in\N$. Since $\widehat T_n^*(\omega,\omega')$ depends on both the original sample $\omega$ and the outcome $\omega'$ of the additional independent random experiment, we may regard $\widehat T_n^*$ as a bootstrapped version of $\widehat T_n$. For the formula display (\ref{abstract bootstrap - definition - almost surely - eq}) in the following Definition \ref{abstract bootstrap - definition - almost surely},  note that the mapping $\omega'\mapsto a_n(\widehat T_n^*(\omega,\omega')-\widehat T_n(\omega))$ is $({\cal F}',{\cal B}^\circ)$-measurable for every fixed $\omega\in\Omega$, because $a_n(\widehat T_n^*-\widehat T_n)$ is $(\overline{\cal F},{\cal B}^\circ)$-measurable with $\overline{\cal F}={\cal F}\otimes{\cal F}'$. That is, $a_n(\widehat T_n^*(\omega,\cdot)-\widehat T_n(\omega))$ can be seen as an $(\bE,{\cal B}^\circ)$-valued random variable on $(\Omega',{\cal F}',\pr')$ for every fixed $\omega\in\Omega$.

\begin{definition}\label{abstract bootstrap - definition - almost surely}
{\bf (Bootstrap version almost surely)} We say that $(\widehat T_n^*)$ is almost surely a bootstrap version of $(\widehat T_n)$ w.r.t.\ the convergence in (\ref{abstract bootstrap - eq - 10}) if
\begin{equation}\label{abstract bootstrap - definition - almost surely - eq}
    a_n(\widehat T_n^*(\omega,\cdot)-\widehat T_n(\omega))\,\leadsto^\circ\,\xi\qquad\mbox{in $(\bE,{\cal B}^\circ,\|\cdot\|_{\bE})$},\qquad\mbox{$\pr$-a.e.\ $\omega$}.
\end{equation}
\end{definition}

Next we intend to introduce the notion of bootstrap version in (outer) probability. To this end let the map $P_n:\overline\Omega\times{\cal B}^\circ\rightarrow[0,1]$ be defined by
\begin{equation}\label{def of cond dist empirical distance}
    P_n((\omega,\omega'),A):=P_n(\omega,A):=\pr'\circ\{a_n(\widehat T_n^*(\omega,\cdot)-\widehat T_n(\omega))\}^{-1}[A],\quad (\omega,\omega')\in\overline\Omega,\,A\in{\cal B}^\circ.
\end{equation}
It provides a conditional distribution of $a_n(\widehat T_n^*-\widehat T_n)$ given $\Pi$, where the $(\overline{\cal F},{\cal F})$-measurable map $\Pi:\overline{\Omega}\rightarrow\Omega$ is defined by
\begin{equation}\label{Def Pi}
    \Pi(\omega,\omega'):=\omega.
\end{equation}
This follows from Lemma \ref{reresentation of P n} (with $X(\omega,\omega')=g(\omega,\omega')=a_n(\widehat T_n^*(\omega,\omega')-\widehat T_n(\omega))$ and $Y=\Pi$). Informally, $\Pi(\omega,\omega')$ specifies that part of the realization $(\omega,\omega')$ of the extended random mechanism $\pr\otimes\pr'$ that represents the ``observed data''; see also Remark \ref{remark on P n} below and the discussion preceding it. By definition $P_n$ is a probability kernel from $(\overline{\Omega},\sigma(\Pi))$ to $(\bE,{\cal B}^\circ)$. However, it is directly clear from (\ref{def of cond dist empirical distance}) that $P_n$ can also be seen as a probability kernel from $(\Omega,{\cal F})$ to $(\bE,{\cal B}^\circ)$.

Let $d_{\scriptsize{\rm BL}}^\circ$ denote the bounded Lipschitz distance (defined in (\ref{def bl metric - add on}) in the Appendix \ref{sec weak conv for open ball}) on the set ${\cal M}_1^\circ$ of all probability measures on $(\bE,{\cal B}^\circ)$. Note that a sequence $(\mu_n)\subseteq{\cal M}_1^\circ$ converges weak$^\circ$ly to some $\mu_0\in{\cal M}_1^\circ$ which concentrates on a separable set, if and only if $d_{\scriptsize{\rm BL}}^\circ(\mu_n,\mu_0)\rightarrow 0$; cf.\ Theorem \ref{Portemanteau}. In general the mapping $\omega\mapsto d_{\scriptsize{\rm BL}}^\circ(P_n(\omega,\cdot),\mbox{\rm law}\{\xi\})$ is not necessarily $({\cal F},{\cal B}(\R_+))$-measurable. For this reason we have to use the outer probability in (\ref{abstract bootstrap - definition - in outer probability - eq}). Recall that the outer probability $\pr^{\scriptsize{\sf out}}[S]$ of an arbitrary subset $S\subseteq\Omega$ is defined to be the infimum of $\pr[\overline{S}]$ over all $\overline{S}\in{\cal F}$ with $\overline{S}\supseteq S$.

\begin{definition}\label{abstract bootstrap - definition - in outer probability}
{\bf (Bootstrap version in (outer) probability)} We say that $(\widehat T_n^*)$ is a bootstrap version in outer probability of $(\widehat T_n)$ w.r.t.\ the convergence in (\ref{abstract bootstrap - eq - 10}) if
\begin{equation}\label{abstract bootstrap - definition - in outer probability - eq}
    \lim_{n\to\infty}\pr^{\scriptsize{\sf out}}\big[\big\{\omega\in\Omega:\,d_{\scriptsize{\rm BL}}^\circ(P_n(\omega,\cdot),\mbox{\rm law}\{\xi\})\ge\delta\big\}\big]=\,0\quad\mbox{ for all }\delta>0.
\end{equation}
When $(\bE,\|\cdot\|_{\bE}$) is separable, we may replace in (\ref{abstract bootstrap - definition - in outer probability - eq}) the outer probability $\pr^{\scriptsize{\sf out}}$ by the ordinary probability $\pr$ and we will say that $(\widehat T_n^*)$ is a bootstrap version in probability of $(\widehat T_n)$ w.r.t.\ the convergence in (\ref{abstract bootstrap - eq - 10}).
\end{definition}

The second part of Definition \ref{abstract bootstrap - definition - in outer probability} can be justified as follows. The assumed separability of $(\bE,\|\cdot\|_{\bE})$ implies that ${\cal M}_1^\circ$ is just the set ${\cal M}_1$ of all Borel probability measures on $\bE$ and that $\omega\mapsto P_n(\omega,\cdot)$ can be seen as an $({\cal F},\sigma({\cal O}_{\rm w}))$-measurable mapping from $\Omega$ to ${\cal M}_1$ (cf.\ Lemma \ref{topology 200}); here ${\cal O}_{\rm w}$ refers to the weak topology on ${\cal M}_1$ (cf.\ Remark \ref{remark appendix B coincide if separable}). By the reverse triangle inequality for metrics we also have that the mapping $\mu\mapsto d_{\scriptsize{\rm BL}}(\mu,\mbox{\rm law}\{\xi\})$ is continuous (recall that $d_{\scriptsize{\rm BL}}:=d_{\scriptsize{\rm BL}}^\circ$ is a metric when $(\bE,\|\cdot\|_{\bE})$ is separable) and thus $(\sigma({\cal O}_{\rm w}),{\cal B}(\R_+))$-measurable. It follows that the mapping $\omega\mapsto d_{\scriptsize{\rm BL}}(P_n(\omega,\cdot),\mbox{\rm law}\{\xi\})$ is $({\cal F},{\cal B}(\R_+))$-measurable.

As our interest lies in deriving bootstrap results for functionals $f$ of $\widehat{T}_n^*$ from bootstrap results for $\widehat{T}_n^*$ itself, we introduce some more notation and restate Definition \ref{abstract bootstrap - definition - in outer probability} for $f(\widehat{T}_n^*)$. Let $(\widetilde{\bE},\|\cdot\|_{\widetilde{\bE}})$ be another normed vector space and assume that $\|\cdot\|_{\widetilde{\bE}}$ is separable. In particular, the open-ball $\sigma$-algebra coincides with the Borel $\sigma$-algebra $\widetilde{\cal B}$ on $\widetilde{\bE}$. Denote by $\widetilde{\cal M}_1$ the set of all probability measures on $(\widetilde{\bE},\widetilde{\cal B})$. Let
$$
    f:\V_f\longrightarrow\widetilde{\bE}
$$
be any map defined on some subset $\V_f\subseteq\V$. Assume that $\widehat T_n$ and $\widehat T_n^*$ take values only in $\V_f$ and that $a_n(f(\widehat T_n^*)-f(\widehat T_n))$ is $(\overline{\cal F},\widetilde{\cal B})$-measurable. Moreover let the map $\widetilde P_n:\overline\Omega\times\widetilde{\cal B}\rightarrow[0,1]$ be defined by
\begin{equation}\label{def of cond dist empirical f distance}
    \widetilde P_n((\omega,\omega'),A):=\widetilde P_n(\omega,A):=\pr'\circ\{a_n(f(\widehat T_n^*(\omega,\cdot))-f(\widehat T_n(\omega)))\}^{-1}[A],~ (\omega,\omega')\in\overline\Omega,\, A\in\widetilde{\cal B}.
\end{equation}
It provides a conditional distribution of $a_n(f(\widehat T_n^*)-f(\widehat T_n))$ given $\Pi$, where $\Pi$ is as in (\ref{Def Pi}). This follows from Lemma \ref{reresentation of P n} (with $X(\omega,\omega')=g(\omega,\omega')=a_n(f(\widehat T_n^*(\omega,\omega'))-f(\widehat T_n(\omega)))$ and $Y=\Pi$). By definition $\widetilde P_n$ is a probability kernel from $(\overline{\Omega},\sigma(\Pi_n))$ to $(\widetilde\bE,\widetilde{\cal B})$. However, it is directly clear from (\ref{def of cond dist empirical f distance}) that $\widetilde P_n$ can also be seen as a probability kernel from $(\Omega,{\cal F})$ to $(\widetilde\bE,\widetilde{\cal B})$. Finally assume that
\begin{equation}\label{abstract bootstrap - eq - 20}
    a_n(f(\widehat T_n)-f(\theta))\,\leadsto\,\widetilde\xi\qquad\mbox{in $(\widetilde{\bE},\widetilde{\cal B},\|\cdot\|_{\widetilde{\bE}})$}
\end{equation}
for some $(\widetilde{\bE},\widetilde{\cal B})$-valued random variable $\widetilde{\xi}$ 
and let $\widetilde d_{\scriptsize{\rm BL}}$ denote the bounded Lipschitz distance on $\widetilde{\cal M}_1$ as defined in (\ref{def bl metric - add on}).

\begin{definition}\label{abstract bootstrap - definition - in outer probability f}
{\bf (Bootstrap version in probability)} We say that $(f(\widehat T_n^*))$ is a bootstrap version in probability of $(f(\widehat T_n))$ w.r.t.\ the convergence in (\ref{abstract bootstrap - eq - 20}) if
\begin{equation}\label{abstract bootstrap - definition - in outer probability f - eq}
    \lim_{n\to\infty}\pr\big[\big\{\omega\in\Omega:\,\widetilde d_{\scriptsize{\rm BL}}(\widetilde P_n(\omega,\cdot),\mbox{\rm law}\{\widetilde\xi\})\ge\delta\big\}\big]=\,0\quad\mbox{ for all }\delta>0.
\end{equation}
\end{definition}

Note that the mapping $\omega\mapsto\widetilde d_{\scriptsize{\rm BL}}(\widetilde P_n(\omega,\cdot),\mbox{\rm law}\{\widetilde\xi\})$ is $({\cal F},{\cal B}(\R_+))$-measurable. Indeed, one can argue as subsequent to Definition \ref{abstract bootstrap - definition - in outer probability}, because we assumed that $(\widetilde\bE,\|\cdot\|_{\widetilde\bE})$ is separable.

\begin{remarknorm}
Note that (\ref{abstract bootstrap - definition - in outer probability f - eq}) implies that (\ref{abstract bootstrap - definition - in outer probability f - eq}) still holds when the bounded Lipschitz distance $\widetilde d_{\scriptsize{\rm BL}}$ is replaced by any other metric on $\widetilde{\cal M}_1$ which generates the weak topology. When $(\bE,\|\cdot\|_{\bE})$ is separable, then the same is true for (\ref{abstract bootstrap - definition - in outer probability - eq}).
{\hspace*{\fill}$\Diamond$\par\bigskip}
\end{remarknorm}

We conclude this section with some comments on the probability kernel $P_n$ defined in (\ref{def of cond dist empirical distance}). As mentioned above, it is a conditional distribution of $a_n(\widehat T_n^*-\widehat T_n)$ given $\Pi$, where to some extent $\Pi(\omega,\omega')=\omega$ can be seen as the ``observable'' sample. On the other hand, for technical reasons the sample space $\Omega$ is often so complex so that only a portion $\Pi_n(\omega)$ of an element $\omega\in\Omega$ can indeed be ``observed''. For instance, when the sample space is an infinite product space, i.e.\ $(\Omega,{\cal F})=(S^\N,{\cal S}^{\otimes\N})$ for some measurable space $(S,{\cal S})$, then de facto one can only observe a finite-dimensional sample, say the first $n$ coordinates $(\omega_1,\ldots,\omega_n)$ of the infinite-dimensional sample $\omega=(\omega_1,\omega_2,\ldots)\in S^\N$. In this case it is obviously appealing to interpret $P_n$ as a conditional distribution of $a_n(\widehat T_n^*-\widehat T_n)$ given $\Pi_n$, where $\Pi_n:S^\N\times\Omega'\rightarrow S^n$ is given by
\begin{equation}\label{Def Pi n}
   \Pi_n((\omega_1,\omega_2,\ldots),\omega'):=(\omega_1,\ldots,\omega_n).
\end{equation}
Under additional mild assumptions this is indeed possible. This follows from the next Remark \ref{remark on P n} if we take there $\Pi_n$ as given in (\ref{Def Pi n}) and $(\Omega^{(n)},{\cal F}^{(n)})$ equal to $(S^n,{\cal S}^{\otimes n})$. Analogously one can regard $\widetilde P_n$ defined in (\ref{def of cond dist empirical f distance}) as a conditional distribution of $a_n(f(\widehat T_n^*)-f(\widehat T_n))$ given $\Pi_n$.

\begin{remarknorm}\label{remark on P n}
Let $(\Omega^{(n)},{\cal F}^{(n)})$ be a measurable space and $\Pi_n:\overline\Omega\rightarrow\Omega^{(n)}$ be an $(\overline{\cal F},{\cal F}^{(n)})$-measurable map for every $n\in\N$. Assume that for every $n\in\N$ the value $\Pi_n(\omega,\omega')$ depends only on $\omega$ and that there exist maps $\tau_n:\Omega^{(n)}\rightarrow\V$ and $\tau_n^*:\Omega^{(n)}\times\Omega'\rightarrow\V$ such that
\begin{equation}\label{def of tau star}
    \tau_n(\Pi_n(\omega,\omega'))=\widehat T_n(\omega)\quad\mbox{ and }\quad\tau_n^*(\Pi_n(\omega,\omega'),\omega')=\widehat T_n^{*}(\omega,\omega')\quad\mbox{ for all }\omega\in\Omega,\,\omega'\in\Omega'
\end{equation}
and
\begin{equation}\label{def of g n}
    g_n(\omega^{(n)},\omega'):=a_n\big(\tau_n^*(\omega^{(n)},\omega')-\tau_n(\omega^{(n)})\big),\qquad(\omega^{(n)},\omega')\in\Omega^{(n)}\times\Omega'
\end{equation}
provides an $({\cal F}^{(n)}\otimes{\cal F}',{\cal B}^\circ)$-measurable map $g_n:\Omega^{(n)}\times\Omega'\rightarrow\bE$. (This implies in particular that $a_n(\widehat T_n^*-\widehat T_n)$ takes values only in $\bE$ and is $(\overline{\cal F},{\cal B}^\circ)$-measurable). Then the map $P_n:\overline\Omega\times{\cal B}^\circ\rightarrow[0,1]$ defined by (\ref{def of cond dist empirical distance}) provides a conditional distribution of $a_n(\widehat T_n^*-\widehat T_n)$ given $\Pi_n$. This follows again from Lemma \ref{reresentation of P n} (with $X(\omega,\omega')=a_n(\widehat T_n^*(\omega,\omega')-\widehat T_n(\omega))$, $Y=\Pi_n$, and $g=g_n$).
{\hspace*{\fill}$\Diamond$\par\bigskip}
\end{remarknorm}


\section{Abstract delta-method for the bootstrap}\label{Delta-Method for the bootstrap}

Theorem \ref{modified delta method for the bootstrap} below establishes an abstract delta-method for the bootstrap for quasi-Hadamard differentiable maps. It uses the notation and definitions introduced in Section \ref{Auxiliaries}. More precisely, let $\V$, $(\bE,\|\cdot\|_\bE)$, $(\Omega,{\cal F},\pr)$, $(\Omega',{\cal F}',\pr')$, $(\overline{\Omega},\overline{{\cal F}},\overline{\pr})$, $\widehat T_n$, $\widehat T_n^*$, $P_n$, $f$, $\V_f$, $(\widetilde\bE,\|\cdot\|_{\widetilde\bE})$, and $\widetilde P_n$ be as in Section \ref{Auxiliaries}. As before assume that $(\widetilde\bE,\|\cdot\|_{\widetilde\bE})$ is separable, and that $\widehat T_n$ and $\widehat T_n^*$ take values only in $\V_f$.

As already discussed in the introduction, in statistical applications the role of $\widehat T_n$ is often played by the empirical distribution function of $n$ identically distributed random variables (sample), so that the plug-in estimator $f(\widehat T_n)$ can be represented as a function of the sample. This special case will be studied in detail in Section \ref{application to statistical functionals}. Due to the measurability problems discussed in the introduction we work with the open-ball $\sigma$-algebra ${\cal B}^\circ$ in our general setting. This is different from the conventional functional delta-method for the bootstrap in the form of \mbox{\citet[Theorems 3.9.11 and 3.9.13]{van der Vaart Wellner 1996}} where the measurability problem is overcome by using the concept of convergence in distribution in the Hoffmann-J{\o}rgensen sense. Moreover, compared to the conventional functional delta-method we work with a weaker notion of differentiability, namely with {\em quasi}-Hadamard differentiability. This kind of differentiability was introduced by \mbox{\cite{BeutnerZaehle2010}} and is recalled in Definition \ref{definition quasi hadamard} in the Appendix.

\begin{theorem}\label{modified delta method for the bootstrap}
{\bf (Delta-method for the bootstrap)} Let $\theta\in\V_f$. Let $\bE_0\subseteq\bE$ be a separable subspace and assume that $\bE_0\in {\cal B}^{\circ}$. Let $(a_n)$ be a sequence of positive real numbers tending to $\infty$, and consider the following conditions:
\begin{itemize}
    \item[(a)] $a_n(\widehat T_n-\theta)$ takes values only in $\bE$, is $({\cal F},{\cal B}^{\circ})$-measurable, and satisfies
    \begin{equation}\label{modified delta method for the bootstrap - assumption - 10}
         a_n(\widehat T_n-\theta)\,\leadsto^\circ\,\xi\qquad\mbox{in $(\bE,{\cal B}^{\circ},\|\cdot\|_{\bE})$}
    \end{equation}
    for some $(\bE,{\cal B}^{\circ})$-valued random variable $\xi$ on some probability space $(\Omega_0,{\cal F}_0,\pr_0)$ with $\xi(\Omega_0)\subseteq\bE_0$.
    \item[(b)] The map $f(\widehat T_n):\Omega\rightarrow\widetilde\bE$ is $({\cal F},\widetilde{\cal B})$-measurable.
    \item[(c)] The map $f$ is quasi-Hadamard differentiable at $\theta$ tangentially to $\bE_0\langle\bE\rangle$ with quasi-Hadamard derivative $\dot f_\theta$ in the sense of Definition \ref{definition quasi hadamard}.
    \item[(d)] The quasi-Hadamard derivative $\dot f_\theta$ can be extended from $\bE_0$ to $\bE$ such that the extension $\dot f_\theta:\bE\rightarrow\widetilde{\bE}$ is linear and $({\cal B}^{\circ},\widetilde{\cal B})$-measurable. Moreover, the extension $\dot f_\theta:\bE\rightarrow\widetilde{\bE}$ is continuous at every point of $\bE_0$.
    \item[(e)] The map $f(\widehat T_{n}^*):\overline\Omega\rightarrow\widetilde\bE$ is $(\overline{\cal F},\widetilde{\cal B})$-measurable.
    \item[(f)] $a_n(\widehat T_n^*-\theta)$ and $a_n(\widehat T_n^{*}-\widehat T_n)$ take values only in $\bE$ and are $(\overline{\cal F},{\cal B}^\circ)$-measurable, and $(\widehat T_n^{*})$ is almost surely a bootstrap version of $(\widehat T_n)$ w.r.t.\ the convergence in (\ref{modified delta method for the bootstrap - assumption - 10}) in the sense of Definition \ref{abstract bootstrap - definition - almost surely}. The latter means that
        \begin{equation}\label{modified delta method for the bootstrap - assumption - 15}
            a_n(\widehat T_n^{*}(\omega,\cdot)-\widehat T_n(\omega))\,\leadsto^\circ\,\xi\qquad\mbox{in $(\bE,\mathcal{B}^{\circ},\|\cdot\|_{\bE})$},\qquad\mbox{$\pr$-a.e.\ $\omega$}.
        \end{equation}
    \item[(f')] $a_n(\widehat T_n^*-\theta)$ and $a_n(\widehat T_n^{*}-\widehat T_n)$ take values only in $\bE$ and are $(\overline{\cal F},{\cal B}^\circ)$-measurable, and $(\widehat T_n^{*})$ is a bootstrap version in outer probability of $(\widehat T_n)$ w.r.t.\ the convergence in (\ref{modified delta method for the bootstrap - assumption - 10}) in the sense of Definition \ref{abstract bootstrap - definition - in outer probability}. The latter means that
        \begin{equation}\label{modified delta method for the bootstrap - assumption - 18}
            \lim_{n\to\infty}\pr^{\scriptsize{\sf out}}\big[\big\{\omega\in\Omega:\,d_{\scriptsize{\rm BL}}^\circ(P_n(\omega,\cdot),\mbox{\rm law}\{\xi\})\ge\delta\big\}\big]=\,0\quad\mbox{ for all }\delta>0.
        \end{equation}
\end{itemize}
Then the following assertions hold:
\begin{itemize}
    \item[(i)] If conditions (a)--(c) hold, then $a_n(f(\widehat T_n)-f(\theta))$ and $\dot f_\theta(\xi)$ are respectively $({\cal F},\widetilde{\cal B})$- and $({\cal F}_0,\widetilde{\cal B})$-measurable, and
    \begin{equation}\label{modified delta method for the bootstrap - assumption - 30}
        a_n(f(\widehat T_n)-f(\theta))\,\leadsto\,\dot f_\theta(\xi)\qquad\mbox{in $(\widetilde\bE,\widetilde{\cal B},\|\cdot\|_{{\widetilde\bE}})$}.
    \end{equation}
    \item[(ii)] If conditions (a)--(f) hold, then $a_n(f(\widehat T_n^*)-f(\widehat T_n))$ and $\dot f_\theta(\xi)$ are respectively $(\overline{\cal F},\widetilde{\cal B})$- and $({\cal F}_0,\widetilde{\cal B})$-measurable, and $(f(\widehat T_n^*))$ is a bootstrap version in probability of $(f(\widehat T_n))$ w.r.t.\ the convergence in (\ref{modified delta method for the bootstrap - assumption - 30}) in the sense of Definition \ref{abstract bootstrap - definition - in outer probability f}. The latter means that
        \begin{equation}\label{modified delta method for the bootstrap - assumption - 40}
            \lim_{n\to\infty}\pr\big[\big\{\omega\in\Omega:\,\widetilde{d}_{\scriptsize{\rm BL}}\big(\widetilde P_n(\omega,\cdot),\mbox{\rm law}\{\dot f_\theta(\xi)\}\big)\ge\delta\big\}\big]=\,0\quad\mbox{ for all }\delta>0.
        \end{equation}
    \item[(iii)]  Assertion (ii) still holds when assumption (f) is replaced by (f').
\end{itemize}
\end{theorem}

Recall that $(\bE,\|\cdot\|_\bE)$ was {\em not} assumed to be separable, so that the mapping $\omega\mapsto d_{\scriptsize{\rm BL}}^\circ(P_n(\omega,\cdot),\mbox{\rm law}\{\xi\})$ is not necessarily $({\cal F},{\cal B}(\R_+))$-measurable. Further note that the Counterexample 1.9.4 in \cite{van der Vaart Wellner 1996} (where $\pr^{\scriptsize{\sf out}}[|\xi_n-0|\ge\delta]=\pr[\eins_{B_n}^{\scriptsize{\sf out}}\ge\delta]=1$ obviously holds for every $n\in\N$ and $\delta\in(0,1)$, with $\xi_n:=\eins_{B_n}$) shows that in general $\pr$-a.s.\ convergence of a sequence $(\xi_n)$ of {\em non}-$({\cal F},{\cal B}(\R))$-measurable functions $\xi_n:\Omega\rightarrow\R$ does not imply convergence in outer probability of $(\xi_n)$. In particular it is not clear to us whether or not condition (f) implies condition (f'). For that reason we consider both conditions separately.

Note that in contrast to the conventional functional delta-method in the form of \mbox{\citet[Theorems 3.9.11 and 3.9.13]{van der Vaart Wellner 1996}} condition (a) of Theorem \ref{modified delta method for the bootstrap} does not involve convergence in distribution in the Hoffmann-J{\o}rgensen sense (based on outer integrals) and condition (f) of Theorem \ref{modified delta method for the bootstrap} does not involve the concept of convergence in outer probability. Thus assertion (ii) of Theorem \ref{modified delta method for the bootstrap} shows in particular that a comprehensive version of the functional delta-method for the bootstrap can be stated without using the concepts of outer integrals and outer probabilities.  Indeed, (part (ii) of) Theorem \ref{modified delta method for the bootstrap} in the form of (part (ii) of) Corollary \ref{modified delta method for the bootstrap - II} below (together with Lemmas \ref{measurability of tau n *} and \ref{measurability of tau n * - beta mixing}) covers plenty of classical plug-in estimators.


\section{Application to plug-in estimators of statistical functionals}\label{application to statistical functionals}

Let $\bD$ be the space of all \cadlag\ functions $v$ on $\R$ with finite sup-norm $\|v\|_\infty:=\sup_{t\in\R}|v(t)|$, and ${\cal D}$ be the $\sigma$-algebra on $\bD$ generated by the one-dimensional coordinate projections $\pi_t$, $t\in\R$, given by $\pi_t(v):=v(t)$. Let $\phi:\R\rightarrow[1,\infty)$ be a weight function, i.e.\ a continuous function being non-increasing  on $(-\infty,0]$ and non-decreasing on $[0,\infty)$. Let $\bD_{\phi}$ be the subspace of $\bD$ consisting of all $x\in\bD$ satisfying $\|x\|_\phi:=\|x\phi\|_\infty<\infty$ and $\lim_{|t|\to\infty} |x(t)|=0$. The latter condition automatically holds when $\lim_{|t|\to\infty} \phi(t)=\infty$. Let ${\cal D}_{\phi}:={\cal D}\cap\bD_{\phi}$ be the trace $\sigma$-algebra on $\bD_{\phi}$. The $\sigma$-algebra on $\bD_{\phi}$ generated by the $\|\cdot\|_\phi$-open balls will be denoted by ${\cal B}_{\phi}^\circ$. The following lemma shows that it coincides with ${\cal D}_{\phi}$.

\begin{lemma}\label{lemma proj eq ball}
${\cal D}_{\phi}={\cal B}_{\phi}^\circ$.
\end{lemma}

\begin{proof}
Without of loss of generality we assume $\lim_{|t|\to\infty}\phi(t)=\infty$. We denote by $B_r(x)$ the $\|\cdot\|_\phi$-open ball around $x\in\bD_{\phi}$ with radius $r$, that is, $B_r(x):=\{y\in\bD_{\phi}:\|x-y\|_\phi<r\}$. On the one hand, for every $t\in\R$ and $a\in\R$ we have
\begin{equation}\label{BallBorelDudley-PROOF-10}
    \pi_t^{-1}((a/\phi(t),\infty))\,=\,\{x\in\bD_{\phi}:\,x(t)>a/\phi(t)\}\,=\,\bigcup_{n\in\N} B_n(x_n),
\end{equation}
where $x_{n}=x_{n,t,a}$ is defined by $x_{n}(s):=(a+(n+1/n)\eins_{[t,t+1/n)}(s))/\phi(s)$. Thus, $\pi_t^{-1}((b,\infty))$ lies in ${\cal B}_\phi^\circ$ for every $t\in\R$ and $b\in\R$. That is, $\pi_t$ is $({\cal B}_{\phi}^\circ,{\cal B}(\R))$-measurable. Hence, ${\cal D}_{\phi}\subseteq{\cal B}_{\phi}^\circ$. On the other hand, any open ball $B_r(x)$ can be represented as
$$
    B_r(x) = \bigcap_{t\in\Q}\{y\in\bD_{\phi}:|x(t)-y(t)|\phi(t)<r\}=\bigcap_{t\in\Q}\pi_t^{-1}\big((x(t)-r/\phi(t),x(t)+r/\phi(t))\big),
$$
and so it lies in ${\cal D}_{\phi}$. Hence, ${\cal B}_{\phi}^\circ\subseteq{\cal D}_{\phi}$.
\end{proof}

For any given distribution function $F$ on the real line, let $\bC_{\phi,F}\subseteq\bD_{\phi}$ be a $\|\cdot\|_\phi$-separable subspace and assume $\bC_{\phi,F}\in{\cal D}_{\phi}$. Moreover let $f:\bD(f)\rightarrow\R$ be a map defined on a set $\bD(f)$ of distribution functions of finite (not necessarily probability) Borel measures on $\R$. In particular, $\bD(f)\subset\bD$. In the following, $\bD$, $(\bD_{\phi},{\cal D}_{\phi},\|\cdot\|_\phi)$, $\bC_{\phi,F}$, $f$, $\bD(f)$, and $(\R,{\cal B}(\R),|\cdot|)$ will play the roles of $\V$, $(\bE,{\cal B}^{\circ},\|\cdot\|_{\bE})$, $\bE_0$, $f$, $\V_f$, and $(\widetilde\bE,\widetilde{\cal B},\|\cdot\|_{\widetilde\bE})$, respectively.

Let $(\Omega,{\cal F},\pr)$ be a probability space and $F\in\bD(f)$ be the distribution function of a Borel probability measure on $\R$. Let $(X_i)$ be a sequence of identically distributed real-valued random variables on $(\Omega,{\cal F},\pr)$ with distribution function $F$. Let $\widehat F_n:\Omega\rightarrow\bD$ be the empirical distribution function of $X_1,\ldots,X_n$, which will play the role of $\widehat T_n$. It is defined by
\begin{equation}\label{Def EmpDF}
    \widehat F_n:=\frac{1}{n}\sum_{i=1}^n\eins_{[X_i,\infty)}.
\end{equation}
Assume that $\widehat F_n$ takes values only in $\bD(f)$. Let $(\Omega',{\cal F}',\pr')$ be another probability space and set $(\overline\Omega,\overline{\cal F},\overline{\pr}):=(\Omega\times\Omega',{\cal F}\otimes{\cal F}',\pr\otimes\pr')$. Moreover let $\widehat F_n^*:\overline{\Omega}\rightarrow\bD$ be any map; see Section \ref{bootstrap results for empirical processes} for an illustration. Assume that $\widehat F_n^*$ take values only in $\bD(f)$. In the present setting Theorem \ref{modified delta method for the bootstrap} can be reformulated as follows.

\begin{corollary}\label{modified delta method for the bootstrap - II}
Let $F\in\bD(f)$. Let $(a_n)$ be a sequence of positive real numbers tending to $\infty$, and consider the following conditions:
\begin{itemize}
    \item[(a)] $a_n(\widehat F_n-F)$ takes values only in $\bD_{\phi}$ and satisfies
    \begin{equation}\label{modified delta method for the bootstrap - assumption - 10 - II}
         a_n(\widehat F_n-F)\,\leadsto^\circ\,B\qquad\mbox{in $(\bD_{\phi},{\cal D}_{\phi},\|\cdot\|_{\phi})$}
    \end{equation}
    for some $(\bD_{\phi},{\cal D}_{\phi})$-valued random variable $B$ on some probability space $(\Omega_0,{\cal F}_0,\pr_0)$ with $B(\Omega_0)\subseteq\bC_{\phi,F}$.
    \item[(b)] The map $f(\widehat F_n):\Omega\rightarrow\R$ is $({\cal F},{\cal B}(\R))$-measurable.
    \item[(c)] The map $f$ is quasi-Hadamard differentiable at $F$ tangentially to $\bC_{\phi,F}\langle\bD_{\phi}\rangle$ with quasi-Hadamard derivative $\dot f_F$ in the sense of Definition \ref{definition quasi hadamard}.
    \item[(d)] The quasi-Hadamard derivative $\dot f_F$ can be extended from $\bC_{\phi,F}$ to $\bD_{\phi}$ such that the extension $\dot f_F:\bD_{\phi}\rightarrow\R$ is linear and $({\cal D}_{\phi},{\cal B}(\R))$-measurable. Moreover, the extension $\dot f_F:\bD_{\phi}\rightarrow\R$ is continuous at every point of $\bC_{\phi,F}$.
    \item[(e)] The map $f(\widehat F_n^*):\overline\Omega\rightarrow\R$ is $(\overline{\cal F},{\cal B}(\R))$-measurable.
    \item[(f)] $a_n(\widehat F_n^{*}-\widehat F_n)$ takes values only in $\bD_{\phi}$ and is $(\overline{\cal F},{\cal D}_{\phi})$-measurable, and $(\widehat F_n^{*})$ is almost surely a bootstrap version of $(\widehat F_n)$ w.r.t.\ the convergence in (\ref{modified delta method for the bootstrap - assumption - 10 - II}) in the sense of Definition \ref{abstract bootstrap - definition - almost surely}. The latter means that
        \begin{equation}\label{modified delta method for the bootstrap - assumption - 15 - II}
            a_n(\widehat F_n^{*}(\omega,\cdot)-\widehat F_n(\omega))\,\leadsto^\circ\,B\qquad\mbox{in $(\bD_{\phi},{\cal D}_{\phi},\|\cdot\|_{\phi})$},\qquad\mbox{$\pr$-a.e.\ $\omega$}.
        \end{equation}
    \item[(f')] $a_n(\widehat F_n^{*}-\widehat F_n)$ takes values only in $\bD_{\phi}$ and is $(\overline{\cal F},{\cal D}_{\phi})$-measurable, and $(\widehat F_n^{*})$ is a bootstrap version in outer probability of $(\widehat F_n)$ w.r.t.\ the convergence in (\ref{modified delta method for the bootstrap - assumption - 10 - II}) in the sense of Definition \ref{abstract bootstrap - definition - in outer probability}. The latter means that (with $P_n$ defined as in (\ref{def of cond dist empirical distance}))
        \begin{equation}\label{modified delta method for the bootstrap - assumption - 18 - II}
            \lim_{n\to\infty}\pr^{\scriptsize{\sf out}}\big[\big\{\omega\in\Omega:\,d_{\scriptsize{\rm BL}}^\circ(P_n(\omega,\cdot),\mbox{\rm law}\{B\})\ge\delta\big\}\big]=\,0\quad\mbox{ for all }\delta>0,
        \end{equation}
\end{itemize}
Then the following assertions hold:
\begin{itemize}
    \item[(i)] If conditions (a)--(c) hold, then $a_n(f(\widehat F_n)-f(\theta))$ and $\dot f_F(B)$ are respectively $({\cal F},{\cal B}(\R))$- and $({\cal F}_0,{\cal B}(\R))$-measurable, and
    \begin{equation}\label{modified delta method for the bootstrap - assumption - 30 - II}
        a_n(f(\widehat F_n)-f(F))\,\leadsto\,\dot f_F(B)\qquad\mbox{in $(\R,{\cal B}(\R))$}.
    \end{equation}
    \item[(ii)] If conditions (a)--(f) hold, then $a_n(f(\widehat F_n^*)-f(\widehat F_n))$ and $\dot f_F(B)$ are respectively $(\overline{\cal F},{\cal B}(\R))$- and $({\cal F}_0,{\cal B}(\R))$-measurable, and $(f(\widehat F_n^*))$ is a bootstrap version in probability of $(f(\widehat F_n))$ w.r.t.\ the convergence in (\ref{modified delta method for the bootstrap - assumption - 30 - II}) in the sense of Definition \ref{abstract bootstrap - definition - in outer probability f}. The latter means that (with $\widetilde P_n$ defined as in (\ref{def of cond dist empirical f distance}))
        $$
            \lim_{n\to\infty}\pr\big[\big\{\omega\in\Omega:\,\widetilde d_{\scriptsize{\rm BL}}\big(\widetilde P_n(\omega,\cdot),\mbox{\rm law}\{\dot f_F(B)\}\big)\ge\delta\big\}\big]=\,0\quad\mbox{ for all }\delta>0.
        $$
    \item[(iii)]  Assertion (ii) still holds when assumption (f) is replaced by (f').
\end{itemize}
\end{corollary}

\begin{proof}
Corollary \ref{modified delta method for the bootstrap - II} is a consequence of Theorem \ref{modified delta method for the bootstrap}, because the measurability assumption in condition (a) and the first measurability assumption of condition (f) (respectively (f')) of Theorem \ref{modified delta method for the bootstrap} are automatically satisfied in the present setting. Indeed, $a_n(\widehat F_n-F)$ is easily seen to be $({\cal F},{\cal D}_{\phi})$-measurable, and the sum of two $(\overline{\cal F},{\cal D}_{\phi})$-measurable maps is clearly $(\overline{\cal F},{\cal D}_{\phi})$-measurable and we assumed here (through(f) (respectively (f'))) that $a_n(\widehat F_n^*-\widehat F_n)$ is $(\overline{\cal F},{\cal D}_{\phi})$-measurable.
\end{proof}

Conditions (e)--(f') of Corollary \ref{modified delta method for the bootstrap - II} will be illustrated in Sections \ref{bootstrap results for empirical processes - iid}--\ref{bootstrap results for empirical processes - beta mixing}. The following examples illustrate conditions (a)--(d) of Corollary \ref{modified delta method for the bootstrap - II}. See also Section \ref{bootstrap results for empirical processes - applications} for specific applications.

\begin{examplenorm}{\bf (for condition (a))}\label{example iid data}
Assume that $X_1,X_2,\dots$ are i.i.d.\ with distribution function $F$, and let $\phi$ be a weight function. If $\int\phi^2dF<\infty$, then Theorem 6.2.1 in \cite{ShorackWellner1986} shows that
$$
    \sqrt{n}(\widehat F_n-F)\,\leadsto^\circ\,B_{F}\qquad\mbox{in $(\bD_{\phi},{\cal D}_{\phi},\|\cdot\|_\phi)$},
$$
where $B_{F}$ is an $F$-Brownian bridge, i.e.\ a centered Gaussian process with covariance function $\Gamma(t_0,t_1)=F(t_0\wedge t_1)(1-F(t_0\vee t_1))$. Note that $B_{F}$ jumps where $F$ jumps and that $\lim_{|t|\to\infty}B_{F}(t)=0$. Thus, $B_{F}$ takes values only in the set $\bC_{\phi,F}\subset\bD_{\phi}$ consisting of all $x\in\bD_{\phi}$ whose discontinuities are also discontinuities of $F$. It was shown in \citet[Corollary B.4]{Kraetschmeretal2013} that the set $\bC_{\phi,F}$ is $\|\cdot\|_{\phi}$-separable and contained in ${\cal D}_{\phi}$.
{\hspace*{\fill}$\Diamond$\par\bigskip}
\end{examplenorm}

\begin{examplenorm}{\bf (for condition (a))}\label{example weakly dependent data}
Let $\phi$ be any weight function, $(X_i)$ be strictly stationary and $\beta$-mixing with distribution function $F$, and assume that $\ex[\phi(X_1)^p]<\infty$ for some $p>2$ and that the mixing coefficients satisfy $\beta_n=o(n^{-p/(p-2)}(\log n)^{2(p-1)/(p-2)})$. Then
$$
    \sqrt{n}(\widehat F_{n}-F)\,\leadsto^\circ\,\widetilde B_{F}\qquad\mbox{in $(\bD_{\phi},{\cal D}_{\phi},\|\cdot\|_{\phi})$},
$$
where $\widetilde B_F$ is a centered Gaussian process with covariance function $\Gamma(t_0,t_1)=F(t_0\wedge t_1)(1-F(t_0\vee t_1))+\sum_{i=0}^1\sum_{k=2}^{\infty}\covi(\eins_{\{X_1 \le t_i\}}, \eins_{\{X_k \le t_{1-i}\}})$. The result follows by verifying the assumptions of Theorem 2.1 in \citet{ArconesYu1994}. We will verify these assumptions in the proof of Theorem \ref{bootstrap results of Radulovic} below. Note that $\widetilde B_{F}$ jumps where $F$ jumps and that $\lim_{|x|\to\infty}\widetilde B_{F}(x)=0$. Thus, $\widetilde B_{F}$ takes values only in the $\|\cdot\|_{\phi}$-separable and ${\cal D}_{\phi}$-measurable set $\bC_{\phi,F}$ introduced in Example \ref{example iid data}. For illustration, note that many GARCH processes are strictly stationary and $\beta$-mixing; see, for instance, \citet[Chapter 3]{FrancqZakoian2010} and \cite{Boussamaetal2011}.
{\hspace*{\fill}$\Diamond$\par\bigskip}
\end{examplenorm}

Further examples for condition (a) can be found in \cite{BeutnerZaehle2010,BeutnerZaehle2012,BeutnerZaehle2013}, \cite{BeutnerWuZaehle2012}, and \cite{Buchsteiner2015}.

\begin{examplenorm}{\bf (for condition (b))}\label{Example DR Functional}
Let $g$ be a continuous concave distortion function as introduced before (\ref{def dist risk meas}). For every real-valued random variable $X$ (on some given atomless probability space) satisfying $\int_0^\infty g(1-F_{|X|}(x)\big)\,dx<\infty$ the {\em distortion risk measure} associated with $g$ is defined by $\rho_g(X):=f_g(F_X)$ with $f_g$ as in (\ref{def dist risk meas}). Here $F_X$ and $F_{|X|}$ denote the distribution functions of $X$ and $|X|$, respectively. The set ${\cal X}_g$ of all random variables $X$ satisfying the above integrability condition provides a linear subspace of $L^1$; this follows from \citet[Proposition 9.5]{Denneberg1994} and \citet[Proposition 4.75]{FoellmerSchied2011}. It is known that $\rho_g$ is a law-invariant coherent risk measure; see, for instance, \cite{WangDhaene1998}. If specifically $g(s)=(s/\alpha)\wedge 1$ for any fixed $\alpha\in(0,1)$, then we have ${\cal X}_g=L^1$ and $\rho_g$ is nothing but the Average Value at Risk at level $\alpha$.

The {\em risk functional} $f_g:\bD(f_g)\rightarrow\R$ corresponding to $\rho_g$ was already introduced in (\ref{def dist risk meas}), where $\bD(f_g)$ is the set of all distribution functions of the random variables of ${\cal X}_g$. Now, the mapping $\omega\mapsto\widehat F_n(\omega,t)=\frac{1}{n}\sum_{i=1}^n\eins_{[X_i(\omega),\infty)}(t)$ is $({\cal F},{\cal B}(\R))$-measurable for every $t\in\R$. Due to the monotonicity of $g$ also the mapping  $\omega\mapsto g(\widehat F_n(\omega,t))$ is $({\cal F},{\cal B}(\R))$-measurable for every $t\in\R$. By the right-continuity of the mapping $t\mapsto g(\widehat F_n(\omega,t))$ for every fixed $\omega\in\Omega$ we obtain in particular that the mapping $(\omega,t)\mapsto g(\widehat F_n(\omega,t))$ is $({\cal F}\otimes{\cal B}(\R),{\cal B}(\R))$-measurable. Fubini's theorem then implies that the mapping $\omega\mapsto f_g(\widehat F_n(\omega,\cdot))$ is $({\cal F},{\cal B}(\R))$-measurable. So we have in particular that condition (b) of Corollary \ref{modified delta method for the bootstrap - II} holds.
{\hspace*{\fill}$\Diamond$\par\bigskip}
\end{examplenorm}

\begin{examplenorm}{\bf (for conditions (c)--(d))}\label{Example HD of DR Functional}
Let $f_g:\bD(f_g)\rightarrow\R$ be as in Example \ref{Example DR Functional}. Let $F\in\bD(f_g)$ with $0<F(\cdot)<1$, and $\phi$ be a weight function satisfying the integrability condition
\begin{equation}\label{Example HD of DR Functional - Int Cond}
    \int_{-\infty}^\infty\frac{g(\gamma F(t))}{F(t)\,\phi(t)}\,dt\,<\,\infty\qquad\mbox{for some }\gamma\in(0,1).
\end{equation}
Assume that the set of points $t\in\R$ for which $g$ is not differentiable at $F(t)$ has Lebesgue measure zero. Then Theorem 2.7 in \cite{Kraetschmeretal2013} shows that the functional $f_g$ is quasi-Hadamard differentiable at $F$ tangentially to $\bC_{\phi,F}\langle\bD_{\phi}\rangle$ with quasi-Hadamard derivative $\dot f_{g;F}:\bC_{\phi,F}\rightarrow\R$ given by
$$
   \dot f_{g;F}(x)\,:=\,\int_{-\infty}^\infty g'(F(t))\,x(t)\,dt,\qquad x\in\bC_{\phi,F},
$$
where $g'$ denotes the right-sided derivative of $g$ and $\bC_{\phi,F}$ is as in Example \ref{example iid data}. Recall that $\bC_{\phi,F}$ is $\|\cdot\|_\phi$-separable and contained in ${\cal D}_{\phi}$; cf.\ Corollary B.4 in \cite{Kraetschmeretal2013}. The derivative $\dot f_{g;F}$ can be extended to $\bD_{\phi}$ through
$$
   \dot f_{g;F}(x)\,:=\,\int_{-\infty}^\infty g'(F(t))\,x(t)\,dt,\qquad x\in\bD_{\phi},
$$
and the extension is linear and continuous on $\bD_{\phi}$. The linearity is obvious and the continuity is ensured by part (ii) of Lemma 4.1 in \cite{Kraetschmeretal2013}. Thus, condition (c) of Corollary \ref{modified delta method for the bootstrap - II} holds. Moreover, using arguments as in Example \ref{Example DR Functional}, one can easily show that the extension $\dot f_{g;F}:\bD_{\phi}\rightarrow\R$ is also $({\cal D}_{\phi},{\cal B}(\R))$-measurable. That is, condition (d) of Corollary \ref{modified delta method for the bootstrap - II} holds too.
{\hspace*{\fill}$\Diamond$\par\bigskip}
\end{examplenorm}


\section{Bootstrap results for empirical processes}\label{bootstrap results for empirical processes}

In the following two subsections, we will give examples for bootstrap versions $(\widehat T_n^*)$ of $(\widehat T_n)$ in the sense of Definitions \ref{abstract bootstrap - definition - almost surely} and \ref{abstract bootstrap - definition - in outer probability} in the context of Section \ref{application to statistical functionals}, i.e.\ in the case where $\widehat T_n$ is given by an empirical distribution function $\widehat F_n$ of real-valued random variables. As mentioned in the introduction these examples can be combined with the quasi-Hadamard differentiability of statistical functionals to lead to bootstrap consistency for the corresponding plug-in estimators. Examples include empirical distortion risk measures as well as U- and V-statistics which will be discussed in Section \ref{bootstrap results for empirical processes - applications}.


\subsection{I.i.d.\ observations}\label{bootstrap results for empirical processes - iid}

We will adopt the notation introduced in Section \ref{application to statistical functionals}. In particular, $(X_i)$ will be a sequence of identically distributed real-valued random variables on $(\Omega,{\cal F},\pr)$ with distribution function $F$, and $\widehat F_n$ will be given by (\ref{Def EmpDF}). Let $(W_{ni})$ be a triangular array of nonnegative real-valued random variables on $(\Omega',{\cal F}',\pr')$ such that $(W_{n1},\ldots,W_{nn})$ is an exchangeable random vector for every $n\in\N$, and define the map $\widehat F_n^*:\overline\Omega\rightarrow\bD$ by
\begin{equation}\label{example for tau}
    \widehat F_n^*(\omega,\omega')\,:=\,\frac{1}{n} \sum_{i=1}^n W_{ni} (\omega')\,\eins_{[X_i(\omega),\infty)}.
\end{equation}
Note that the sequence $(X_i)$ and the triangular array $(W_{ni})$ regarded as families of random variables on the product space $(\overline\Omega,\overline{\cal F},\overline\pr):=(\Omega\times\Omega',{\cal F}\otimes{\cal F}',\pr\otimes\pr')$ are independent. Of course, we will tacitly assume that $(\Omega',{\cal F}',\pr')$ is rich enough to host all of the random variables described in (a)--(b) in Theorem \ref{bootstrap results of VdV-W}.

\begin{lemma}\label{measurability of tau n *}
$a_n(\widehat F_n^*-\widehat F_n)$ takes values only in $\bD_{\phi}$ and is $(\overline{\cal F},{\cal D}_{\phi})$-measurable. That is, the first part of condition (f) (respectively (f')) of Corollary \ref{modified delta method for the bootstrap - II} holds true.
\end{lemma}

\begin{proof}
First of all note that $a_n(\widehat F_n^*((\omega,\omega'),t)-\widehat F_n(\omega,t))$ can be written as
$$
    a_n\Big(\frac{1}{n}\sum_{i=1}^n W_{ni}(\omega')\eins_{[X_i(\omega),\infty)}(t)-\frac{1}{n}\sum_{i=1}^n\eins_{[X_i(\omega),\infty)}(t)\Big)=:\,\Xi_n((\omega,\omega'),t)
$$
for all $t\in\R$ and $(\omega,\omega')\in\overline\Omega$. The mapping $(\omega,\omega')\mapsto \Xi_n((\omega,\omega'),t)$ is $(\overline{\cal F},{\cal B}(\R))$-measurable for every $t\in\R$, and the mapping $t\mapsto \Xi_n((\omega,\omega'),t)$ is right-continuous for every $(\omega,\omega')\in\overline\Omega$. It follows that the mapping $(\omega,\omega')\mapsto \Xi_n((\omega,\omega'),\cdot)$ form $\overline\Omega$ to $\bD$ is $(\overline{\cal F},{\cal D})$-measurable. Further, $\Xi_n((\omega,\omega'),\cdot)$ obviously takes values only in $\bD_\phi$ for every $(\omega,\omega')\in\overline\Omega$. Thus $\Xi_n$ can indeed be seen as an $(\overline{\cal F},{\cal D}_{\phi})$-measurable map from $\overline\Omega$ to $\bD_\phi$ ($\subseteq\bD$).
\end{proof}

The proof of the following Theorem \ref{bootstrap results of VdV-W} strongly relies on Section 3.6.2 in \citet{van der Vaart Wellner 1996}. In fact, the elaborations in Section 3.6.2 in \citet{van der Vaart Wellner 1996} yield slightly stronger results compared to those of Theorem \ref{bootstrap results of VdV-W}, because van der Vaart and Wellner work in a more general framework. More precisely, they establish {\em outer} almost sure bootstrap results for the empirical process w.r.t.\ convergence in distribution in the {\em Hoffmann-J{\o}rgensen sense}. The first result on Efron's bootstrap for the empirical process of i.i.d.\ random variables was given by \citet[Theorem 4.1]{BickelFreedman1981} for the uniform sup-norm, that is, for $\phi\equiv 1$. \cite{Gaenssler1986} extended this result to Vapnik--\v{C}ervonenkis classes. For a version of Efron's bootstrap in a very general set-up, see also \citet[Theorem 2.4]{GineZinn1990}.

\begin{theorem}\label{bootstrap results of VdV-W}
Assume that the random variables $X_1,X_2,\ldots$ are i.i.d., their distribution function $F$ satisfies $\int\phi^2dF<\infty$, and one of the following two settings is met.
\begin{itemize}
    \item[(a)] {\bf (Efron's bootstrap)} The random vector $(W_{n1},\ldots,W_{nn})$ is multinomially distributed according to the parameters $n$ and $p_1=\cdots=p_n=\frac{1}{n}$ for every $n\in\N$.
    \item[(b)] {\bf (Bayesian bootstrap)} $W_{ni}=Y_i/\overline{Y}_n$ for every $n\in\N$ and $i=1,\ldots,n$, where $\overline{Y}_n:=\frac{1}{n}\sum_{j=1}^nY_j$ and $(Y_j)$ is any sequence of nonnegative i.i.d.\ random variables on $(\Omega',{\cal F}',\pr')$ with distribution $\mu$ which satisfies $\int_0^\infty{\mu}[(x,\infty)]^{1/2}\,dx<\infty$ and whose standard deviation coincides with its mean and is strictly positive.
\end{itemize}
Then (condition (a) and) the second part of condition (f) of Corollary \ref{modified delta method for the bootstrap - II} hold for $a_n=\sqrt{n}$, $B=B_F$ and $\widehat F_n^*$ defined in (\ref{example for tau}), where $B_F$ is an $F$-Brownian bridge, i.e.\ a centered Gaussian process with covariance function $\Gamma(t_0,t_1)=F(t_0\wedge t_1)\overline{F}(t_0\vee t_1)$.
\end{theorem}

\begin{proof}
The claim of Theorem \ref{bootstrap results of VdV-W} would follow from the second assertion of Theorem 3.6.13 in \citet{van der Vaart Wellner 1996} with ${\cal F}=\F_\phi:=\{\phi(x)\eins_{(-\infty,x]}:x\in\R\}$ if we could show that the assumptions of Theorem 3.6.13 in \citet{van der Vaart Wellner 1996} are fulfilled in each of the settings (a)--(b). At this point we stress the facts that convergence in distribution in the Hoffmann-J{\o}rgensen sense implies convergence in distribution$^\circ$ for the open-ball $\sigma$-algebra and that outer almost sure convergence (as defined in part (iii) of Definition 1.9.1 in \citet{van der Vaart Wellner 1996}) implies almost sure convergence (i.e.\ convergence almost everywhere) in the classical sense. The latter follows from Proposition 1.1 in \citet{Dudley2010}.

In Theorem 3.6.13 in \citet{van der Vaart Wellner 1996} it is assumed that the following three assertions hold:
\begin{itemize}
    \item[1)] $\F_\phi$ is a Donsker class w.r.t.\ $\pr$, and $(t_1,\ldots,t_n)\mapsto\sup_{f\in\F_{\phi,\delta}}|\sum_{i=1}^n\lambda_if(t_i)|$ is a measurable mapping on the completion of $(\R^n,{\cal B}(\R^n),\pr_{X_1}^{\otimes n})$ for every $\delta>0$, $\lambda_1,\ldots,\lambda_n\in\R$ and $n\in\N$. Here we set $\F_{\phi,\delta}:=\{f_1-f_2\,:\,f_1,f_2\in\F_\phi,\,\rho_\pr(f_1-f_2)<\delta\}$ with $\rho_\pr(f):=\vari_\pr[f(X_1)]^{1/2}$, where $\vari_\pr$ refers to the variance w.r.t.\ $\pr$.
    \item[2)] $\ex_\pr^{\sf out}[\overline{f}(X_1)^2]<\infty$ for the envelope function $\overline{f}(t):=\sup_{f\in\F_\phi}(f(t)-\ex_\pr[f(X_1)])$, where $\ex_\pr^{\sf out}$ refers to the outer expectation w.r.t.\ $\pr$.
    \item[3)] $(W_{n1},\ldots,W_{n,n})$ is an exchangeable nonnegative random vector for every $n\in\N$, and the triangular array $(W_{ni})$ satisfies condition (3.6.8) in \citet{van der Vaart Wellner 1996}.
\end{itemize}
We will now verify 1)--3).

1): The assumption $\int\phi^2dF<\infty$ 
ensures that $\F_\phi$ is a Donsker class w.r.t.\ $\pr$; cf.\ Example \ref{example iid data}. To verify the second part of assertion 1), let $\delta>0$ arbitrary but fixed and $f\in\F_{\phi,\delta}$ with $\rho_\pr(f)<\delta$. Now, $f$ has the representation $f=\phi(x_1)\eins_{(-\infty,x_1]}-\phi(x_2)\eins_{(-\infty,x_2]}$ for some $x_1,x_2\in\R$, and
\begin{eqnarray*}
    \rho_\pr(f)
    & = & \vari_\pr\big[\phi(x_1)\eins_{(-\infty,x_1]}(X_1)-\phi(x_2)\eins_{(-\infty,x_2]}(X_1)\big]\\
    & = & \phi(x_1)^2F(x_1)(1-F(x_1))+\phi(x_2)^2F(x_2)(1-F(x_2))\\
    & & -\phi(x_1)\phi(x_2)F(x_1\wedge x_2)(1-F(x_1\wedge x_2))
\end{eqnarray*}
depends (right) continuously on $(x_1,x_2)$. So we can find a sequence $(g_m)$ in the countable subclass $\G_{\phi,\delta}:=\{g_{q_1,q_2}=\phi(q_1)\eins_{(-\infty,q_1]}-\phi(q_2)\eins_{(-\infty,q_2]}:q_1,q_2\in\Q,\,\rho_\pr(g_{q_1,q_2})<\delta\}$ of $\F_{\phi,\delta}$ such that $g_m(t)\rightarrow f(t)$ for every $t\in\R$. For instance, $g_m:=g_{q_{1,m},q_{2,m}}$ for any sequences $(q_{1,m})$ and $(q_{2,m})$ in $\Q$ such that $q_{1,m}\searrow x_1$, $q_{2,m}\searrow x_2$ and $\rho_\pr(g_{q_{1,m},q_{2,m}})<\delta$. As discussed in Example 2.3.4 in \citet{van der Vaart Wellner 1996} this implies that the second part of assertion 1) holds.

2): We first of all note that in the present setting the outer expectation $\ex_\pr^{\sf out}$ can be replaced by the classical expectation $\ex_\pr$ w.r.t.\ $\pr$.  Indeed, the envelope function $\overline{f}$ can be written as
$$
    \overline{f}(t)\,=\,\sup_{x\in\R}(\eins_{(-\infty,x]}(t)-F(x))\phi(x)\,=\,\sup_{q\in\Q}(\eins_{(-\infty,q]}(t)-F(q))\phi(q)
$$
and is thus Borel measurable. So it remains to show $\ex[\overline{f}(X_1)^2]<\infty$. To this end, we note that the assumption $\int\phi^2dF<\infty$ implies
$$
    M_1:=\sup_{t \leq 0} F(t)^2\phi(t)^2<\infty\qquad \mbox{and}\qquad M_2:=\sup_{t > 0} (1-F(t))^2\phi(t)^2<\infty.
$$
Furthermore, for $t \leq 0$ we have
$$
    (\eins_{(-\infty,x]}(t)-F(x))^2\phi(x)^2\,=\,
    \left\{\begin{array}{rcc}
        (1-F(x))^2\phi(x)^2 & , & t\le x\\
        F(x)^2\phi(x)^2 & , &  t>x
    \end{array}
    \right.
$$
and so, since the mapping $x \mapsto (1-F(x))^2 \phi(x)^2$ is non-increasing on $[t,0]$,
$$
    \overline{f}(t)^2\,=\,\sup_{x \in \R}(\eins_{(-\infty,x]}(t)-F(x))^2\phi(x)^2\,\le\,\max\{M_1,(1-F(t))^2\phi(t)^2,M_2\}\,=:\,g(t).
$$
For $t>0$ we obtain similarly
$$
    \overline{f}(t)^2\,=\,\sup_{x \in \R}(\eins_{(-\infty,x]}(t)-F(x))^2\phi(x)^2\, \le\, \max\{M_1,F(t)^2\phi(t)^2, M_2\}\,=:\,g(t),
$$
because the mapping $x \mapsto(F(x) \phi(x))^2$ is non-decreasing on $(0,t]$. Hence, $\ex[\overline{f}(X_1)^2]\le\ex[g(X_1)^2]<\infty$ due to our assumption $\int\phi^2dF<\infty$.

3): Examples 3.6.10 and 3.6.12 in \citet{van der Vaart Wellner 1996} show that assertion 3) holds in each of the settings (a)--(b).
\end{proof}


\subsection{Stationary, $\beta$-mixing observations}\label{bootstrap results for empirical processes - beta mixing}

As in Section \ref{bootstrap results for empirical processes - iid}, we will adopt the notation introduced in Section \ref{application to statistical functionals}. In particular, $(X_i)$ will be a sequence of identically distributed real-valued random variables on $(\Omega,{\cal F},\pr)$ with distribution function $F$, and $\widehat F_n$ will be given by (\ref{Def EmpDF}). Let $(\ell_n)$ be a sequence of integers such that $\ell_n\nearrow\infty$ as $n\rightarrow\infty$, and $\ell_n<n$ for all $n\in\N$. Set $k_n:=\lfloor n/\ell_n\rfloor$ for all $n\in\N$. Let $(I_{nj})_{n\in\N,\,1\le j\le k_n}$ be a triangular array of random variables on $(\Omega',{\cal F}',\pr')$ such that $I_{n1},\ldots,I_{nk_n}$ are i.i.d.\ according to the uniform distribution on $\{1,\ldots,n\}$ for every $n\in\N$. Define the map $\widehat F_n^*:\overline\Omega\rightarrow\bD$ by
\begin{equation}\label{example for tau - beta mixing}
    \widehat F_n^*(\omega,\omega')\,:=\,\frac{1}{n}\sum_{i=1}^{n}W_{ni}(\omega')\eins_{[X_i(\omega),\infty)}
\end{equation}
with
\begin{equation}\label{example for tau - beta mixing - 2}
    W_{ni}(\omega')\, := \, \sum_{j=1}^{k_n}\Big(\eins_{\{I_{nj}\le i\le (I_{nj}+\ell_n-1)\wedge n\}}(\omega')+\eins_{\{I_{nj}+\ell_n-1 > n,\,1 \le i \le I_{nj}+\ell_n-1-n\}}(\omega')\Big).
\end{equation}
Note that, as before, the sequence $(X_i)$ and the triangular array $(W_{ni})$ regarded as families of random variables on the product space $(\overline\Omega,\overline{\cal F},\overline\pr):=(\Omega\times\Omega',{\cal F}\otimes{\cal F}',\pr\otimes\pr')$ are independent.

At an informal level this means that given a sample $X_1,\ldots,X_n$, we pick $k_n$ blocks of length $\ell_n$ in the (artificially) extended sample $X_1,\ldots,X_n,X_{n+1},\ldots,X_{n+\ell_n-1}$ (with $X_{n+i}:=X_i$, $i=1,\ldots,\ell_n-1$) where the start indices $I_{n1},I_{n2},\ldots,I_{nk_n}$ are chosen independently and uniformly in the set of all indices $\{1,\ldots,n\}$:
\begin{center}
\begin{tabular}{ll}
    block $1$: \qquad & $X_{I_{n1}},X_{I_{n1}+1},\ldots,X_{I_{n1}+\ell_n-1}$\\
    block $2$: \qquad & $X_{I_{n2}},X_{I_{n2}+1},\ldots,X_{I_{n2}+\ell_n-1}$\\
    & $\vdots$\\
    block $k_{n}$: \qquad & $X_{I_{nk_n}},X_{I_{nk_n}+1},\ldots,X_{I_{nk_n}+\ell_n-1}$
\end{tabular}
\end{center}
The bootstrapped empirical distribution function $\widehat F_n^*$ is then defined to be the distribution function of the discrete finite (not necessarily probability) measure with atoms $X_1,\ldots,X_n$ carrying masses $W_{n1},\ldots,W_{nn}$ respectively, where $W_{ni}$ specifies the number of blocks which contain $X_i$.

\begin{lemma}\label{measurability of tau n * - beta mixing}
$a_n(\widehat F_n^*-\widehat F_n)$ takes values only in $\bD_{\phi}$ and is $(\overline{\cal F},{\cal D}_{\phi})$-measurable. That is, the first part of condition (f) (respectively (f')) of Corollary \ref{modified delta method for the bootstrap - II} holds true.
\end{lemma}

\begin{proof}
The proof of Lemma \ref{measurability of tau n *} with the obvious modifications also applies to Lemma \ref{measurability of tau n * - beta mixing}.
\end{proof}

The bootstrap method induced by the bootstrapped empirical distribution function $\widehat F_n^*$ defined in (\ref{example for tau - beta mixing})--(\ref{example for tau - beta mixing - 2}) is the so-called {\em circular bootstrap}; see, for instance, \citet{PolitisRomano1992} and \citet{Radulovic1996}. The circular bootstrap is only a slight modification of the moving blocks bootstrap that was independently introduced by \cite{Kuensch1989} in the context of the sample mean and by \citet{Liu1992}. \cite{Buehlmann1994,Buehlmann1995}, \cite{NaikNimbalkarRajarshi1994}, and \cite{Radulovic1996} extended K\"unsch's approach to empirical processes of strictly stationary, mixing observations. \cite{Doukhan(2014)} extended Shao's so-called dependent wild bootstrap for smooth functions of the sample  mean (cf.\ \cite{Shao2010}) to the empirical process of strictly stationary and $\beta$-mixing observations. For an application of the delta-method based on the notion of quasi-Hadamard differentiability the most interesting results are those that allow for weight functions $\phi$ with $\lim_{|x|\rightarrow \infty}\phi(x) \rightarrow \infty$.  The following result is derived from Theorem 1 in \cite{Radulovic1996}.

\begin{theorem}{\bf (Circular bootstrap)}\label{bootstrap results of Radulovic}
Denote by $F$ the distribution function of $X_1$ and assume that the following conditions hold:
\begin{itemize}
    \item[(a)] $\int\phi^p\,dF<\infty$ for some $p>2$.
    \item[(b)] The sequence of random variables $(X_i)$ is strictly stationary and $\beta$-mixing with mixing coefficients $(\beta_i)$ satisfying $\beta_i={\cal O}(i^{-b})$ for some $b>p/(p-2)$.
    \item[(c)] The block length $\ell_n$ satisfies $\ell_n={\cal O}(n^{\gamma})$ for some $\gamma\in(0,\frac{p-2}{2(p-1)})$.
\end{itemize}
Then (condition (a) and) the second part of condition (f') of Corollary \ref{modified delta method for the bootstrap - II} hold for $a_n=\sqrt{n}$, $B=\widetilde B_F$ and $\widehat F_n^*$ defined in (\ref{example for tau - beta mixing}), where $\widetilde B_F$ is a centered Gaussian process with covariance function $\Gamma(t_0,t_1)=F(t_0\wedge t_1)(1-F(t_0\vee t_1))+\sum_{i=0}^1\sum_{k=2}^{\infty}\covi(\eins_{\{X_1 \le t_i\}}, \eins_{\{X_k \le t_{1-i}\}})$.
\end{theorem}

A similar result that allows to verify condition (f) of Corollary \ref{modified delta method for the bootstrap - II} (where in (\ref{modified delta method for the bootstrap - assumption - 15 - II}) the empirical distribution function $\widehat F_n$ is replaced by the conditional expectation of $\widehat F_n^*$) can be found in \citet[Theorem 1]{Buehlmann1995}.

\bigskip

\begin{proof}{\bf of Theorem \ref{bootstrap results of Radulovic}}\,
It was shown in \citet[Theorem 2.1]{ArconesYu1994} that under conditions (a)--(b) of Theorem \ref{bootstrap results of Radulovic} the condition (a) of Corollary \ref{modified delta method for the bootstrap - II} is satisfied; see also Example \ref{example weakly dependent data}. In the following we will show that under assumption (a) of Theorem \ref{bootstrap results of Radulovic} the following two assumptions of Theorem 1 in \cite{Radulovic1996} are met for the class of functions $\F_{\phi}:=\{f_x: x \in \R\}$ with $f_x(\cdot):=\phi(x)\eins_{(-\infty,x]}(\cdot)$ for $x \leq 0$ and $f_x(\cdot):=-\phi(x)\eins_{(x,\infty)}(\cdot)$ for $x>0$:
\begin{itemize}
    \item[1)] $\F_\phi$ is a VC-subgraph class.
    \item[2)] $\int \overline{f}^{\,p} dF<\infty$ for the envelope function $\overline{f}(t):=\sup_{x\in\R}|f_x(t)|$.
\end{itemize}
The other assumptions of Theorem 1 in \cite{Radulovic1996} are just our assumptions (b) and (c). Then, since we may identify the maps $x \mapsto\sqrt{n}(\widehat F_n(x)- F(x))\phi(x)$ and $x \mapsto\sqrt{n}(\widehat F_n^*(x)- \widehat F_n(x))\phi(x)$ with respectively $f_x \mapsto\sqrt{n}(\int f_x d\widehat F_n -\int f_x dF)$ and $f_x \mapsto\sqrt{n}(\int f_x d\widehat F_n^* -\int f_x d\widehat F_n)$, Theorem 1 in \cite{Radulovic1996} implies that condition (f') of Corollary \ref{modified delta method for the bootstrap - II} is satisfied too.

Before verifying 1), let us recall the definition of VC-subgraph class; cf., for instance, \citet[Section 2.6]{van der Vaart Wellner 1996}. First recall that the VC-index of a collection ${\cal C}$ of subsets of a nonempty set $\boldsymbol{Y}$ is defined by $V({\cal C}):=\inf\{n:{\rm m}^{{\cal C}}(n)<2^n\}$ with the convention $\inf\emptyset:=\infty$, where
\begin{eqnarray}\label{eq vc class def}
    {\rm m}^{{\cal C}}(n)\,:=\,\max_{y_1,\ldots,y_n\in\boldsymbol{Y}}\,\#\{C\cap\{y_1,\ldots,y_n\}:\,C\in{\cal C}\}.
\end{eqnarray}
A collection ${\cal C}$ is said to be a {\em VC-class} if $V({\cal C})<\infty$. A class $\F$ of functions $f:\R\rightarrow\R$ is said to be a {\em VC-subgraph class} if the collection ${\cal C}_\F:=\{\{(x,t)\in\R^2:t<f(x)\}:f\in\F\}$ is a VC-class of sets in $\boldsymbol{Y}:=\R^2$.

1): We will show that $\F_{\phi}$ is a VC-subgraph class with $V({\cal C}_{\F_\phi})\le 3$. For $V({\cal C}_{\F_\phi})\le 3$ it suffices to show that ${\rm m}^{{\cal C}_{\F_\phi}}(3)<2^3$. Note that that ${\rm m}^{{\cal C}_{\F_\phi}}(3)<2^3$ means that for every choice of $y_1,y_2,y_3\in\R^2$ there exists at least one of the $2^3$ subsets of $\{y_1,y_2,y_3\}$ which cannot be represented as $C\cap\{y_1,y_2,y_3\}$ for any $C\in{\cal C}_{\F_\phi}$. By way of contradiction assume that there exist $y_1=(x_1,t_1)$, $y_2=(x_2,t_2)$, $y_3=(x_3,t_3)$ in $\R^2$ such that every subset of $\{y_1,y_2,y_3\}$ has the representation $C\cap\{y_1,y_2,y_3\}$ for some $C\in{\cal C}_{\F_\phi}$. Then, in particular, there exist $C_{12}, C_{13}, C_{23} \in {\cal C}_{\F_\phi} $ such that
\begin{eqnarray}\label{vc class in terms of sets}
    C_{12} \cap \{(x_1,t_1),(x_2,t_2),(x_3,t_3)\} & = & \{(x_1,t_1),(x_2,t_2)\}, \nonumber \\
    C_{13} \cap \{(x_1,t_1),(x_2,t_2),(x_3,t_3)\} & = & \{(x_1,t_1),(x_3,t_3)\}, \nonumber \\
    C_{23} \cap \{(x_1,t_1),(x_2,t_2),(x_3,t_3)\} & = & \{(x_2,t_2),(x_3,t_3)\}.
\end{eqnarray}
We may and do assume without loss of generality that $x_1 \le x_2 \le x_3$. Then, if (\ref{vc class in terms of sets}) held true, there would exist $x_{12},x_{13},x_{23}\in\R$ such that
\begin{align}\label{vc class in terms of functions}
    & t_1 <   f_{x_{12}}(x_1),\qquad t_2 <   f_{x_{12}}(x_2),\qquad t_3 \geq f_{x_{12}}(x_3), \nonumber\\
    & t_1 <   f_{x_{13}}(x_1),\qquad t_2 \ge f_{x_{13}}(x_2),\qquad t_3 <    f_{x_{13}}(x_3), \nonumber \\
    & t_1 \ge f_{x_{23}}(x_1),\qquad t_2 <   f_{x_{23}}(x_2),\qquad t_3 <    f_{x_{23}}(x_3).
\end{align}
First assume $x_{12} \leq 0$. In this case we have $f_{x_{12}}(\cdot)=\eins_{(-\infty,x_{12}]}(\cdot)\phi(x_{12})$ and thus $t_3\ge0$ (due to $t_3\ge f_{x_{12}}(x_3)$). But then $f_{x_{13}}$ and $f_{x_{23}}$ are also of the form $f_{x_{13}}(\cdot)=\eins_{(-\infty,x_{13}]}(\cdot)\phi(x_{13})$ and $f_{x_{23}}(\cdot)=\eins_{(-\infty,x_{23}]}(\cdot)\phi(x_{23})$, because $t_3 < f_{x_{13}}(x_3)$, $t_3 < f_{x_{23}}(x_3)$, and functions of the form $f_x(\cdot)=-\eins_{(x,\infty)}\phi(x)$ take values only in $(-\infty,-1]\cup\{0\}$. From the second and the third line of (\ref{vc class in terms of functions}) we can now conclude that $f_{x_{13}}(x_1)=f_{x_{13}}(x_3)$, $x_3\le x_{13}$, and $f_{x_{23}}(x_2)=f_{x_{23}}(x_3)$, $x_3\le x_{23}$, respectively. It follows that
\begin{eqnarray}\label{vc subclass contradiction part 1}
    f_{x_{13}}(x_1)=f_{x_{13}}(x_2)\qquad\mbox{and}\qquad f_{x_{23}}(x_1)=f_{x_{23}}(x_2),
\end{eqnarray}
because $x_2 \le x_3$ (which implies $x_2 \in (-\infty,x_{13}]$) and $x_1 \le x_2$ (which implies $x_1 \in (-\infty,x_{23}]$). On the other hand, by (\ref{vc class in terms of functions}) we obviously have
\begin{eqnarray}\label{vc subclass contradiction part 3}
    f_{x_{13}}(x_1) > f_{x_{23}}(x_1)\qquad \mbox{and}\qquad f_{x_{23}}(x_2) > f_{x_{13}}(x_2).
\end{eqnarray}
But (\ref{vc subclass contradiction part 1}) and (\ref{vc subclass contradiction part 3}) contradict each other.

Now assume $x_{12}>0$. This implies that $f_{x_{12}}$ takes values only in $(-\infty,-1]\cup\{0\}$, and therefore $f_{x_{12}}(x_1)\le0$ and $f_{x_{12}}(x_2)\le0$. It follows that $t_1<0$ and $t_2<0$. The latter two inequalities imply $f_{x_{23}}(x_1)<0$ and $f_{x_{13}}(x_2)<0$, respectively. It follows that $x_{23}>0$ and $x_{13}>0$, because otherwise $f_{x_{23}}$ or $f_{x_{13}}$ would take values only in $\{0\}\cup[1,\infty)$. In particular, $t_3<0$ (since $t_3<f_{x_{23}}(x_3)$). That is, we have $t_1,t_2,t_3<0$ and  $f_{x_{12}}(\cdot)=-\eins_{(x_{12},\infty)}(\cdot)\phi(x_{12})$, $f_{x_{13}}(\cdot)=-\eins_{(x_{13},\infty)}(\cdot)\phi(x_{13})$, $f_{x_{23}}(\cdot)=-\eins_{(x_{23},\infty)}(\cdot)\phi(x_{23})$. From the third line of (\ref{vc class in terms of functions}) we first conclude that $x_1 > x_{23}$, because $t_1<0$ (so that $t_1 \geq f_{x_{23}}(x_1)$ is only possible if $x_1 > x_{23}$). Then we also have $x_2 > x_{23}$ and $x_3 > x_{23}$, because $x_3 \geq x_2 \geq x_1$. This implies  $f_{x_{23}}(x_1)=f_{x_{23}}(x_2)=f_{x_{23}}(x_3)$, and we conclude from the third line of (\ref{vc class in terms of functions}) that $t_1 > t_2$. Similarly, from the second line of (\ref{vc class in terms of functions}) we obtain $t_2 > t_3$. Summarizing we must have
\begin{eqnarray}\label{vc class pick one set}
    0 > t_1 > t_2 > t_3.
\end{eqnarray}
Recall that we assumed (by way of contradiction) that $y_1=(x_1,t_1)$, $y_2=(x_2,t_2)$, $y_3=(x_3,t_3)$ are such that every subset of $\{y_1,y_2,y_3\}$ has the representation $C\cap\{y_1,y_2,y_3\}$ for some $C\in{\cal C}_{\F_\phi}$. In particular, there exists a set $C_{2|1,3} \in \mathcal{C}_{\mathbb{F}_{\phi}}$ with
$$
    C_{2|1,3} \cap \{(x_1,t_1),(x_2,t_2),(x_3,t_3)\}\,=\,\{(x_2,t_2)\}.
$$
That is, there exists some $x_{2|1,3}\in\R$ such that
\begin{eqnarray}\label{vc class pick one set II}
    t_1 \geq f_{x_{2|1,3}}(x_1),\qquad t_2 < f_{x_{2|1,3}}(x_2),\qquad t_3 \geq f_{x_{2|1,3}}(x_3).
\end{eqnarray}
Since $t_1<0$, we must have $x_{2|1,3}>0$ (i.e.\ $f_{x_{2|1,3}}(\cdot)=-\eins_{(x_{2|1,3},\infty)}(\cdot)\phi(x_{2|1,3})$) and $x_1>x_{2|1,3}$. The latter inequality implies in particular $x_2>x_{2|1,3}$ and $x_3>x_{2|1,3}$, because $x_3 \geq x_2 \geq x_1$. Hence $f_{x_{2|1,3}}(x_1)=f_{x_{2|1,3}}(x_2)=f_{x_{2|1,3}}(x_3)$. In view of (\ref{vc class pick one set II}), this gives $t_2 <  t_3$. But this contradicts (\ref{vc class pick one set}).

2): The envelope function $\overline{f}$ is given by $\overline{f}(t)=\phi(t)$ for $t \leq 0$ and by $\overline{f}(t)=\phi(t-)=\phi(t)$ (recall that $\phi$ is continuous) for $t>0$. Then under assumption (a) the integrability condition 2) holds.
\end{proof}


\subsection{Some applications}\label{bootstrap results for empirical processes - applications}

In this section we discuss two specific examples. First we rigorously treat the case of empirical distortion risk measures. Thereafter we informally discuss bootstrap results for U- and V-statistics.

1) Let $f_g:\bD(f_g)\rightarrow\R$ be the distortion risk functional associated with a continuous concave distortion function as in (\ref{def dist risk meas}) and Example \ref{Example DR Functional}, and let $\phi:\R\rightarrow[1,\infty)$ be any continuous function. Let $F\in\bD(f_g)$ satisfy the integrability condition (\ref{Example HD of DR Functional - Int Cond}). Let $(X_i)$ be a strictly stationary sequence of real-valued random variables on some probability space $(\Omega,{\cal F},\pr)$ with distribution function $F$. Let $\widehat F_n$ be the empirical distribution function of $X_1,\ldots,X_n$ defined by (\ref{Def EmpDF}). If $X_1,X_2,\ldots$ are independent, $\int\phi^2\,dF<\infty$, and $\widehat F_n^*$ is as in Theorem \ref{bootstrap results of VdV-W} (on some extension $(\overline\Omega,\overline{\cal F},\overline{\pr})=(\Omega\times\Omega',{\cal F}\otimes{\cal F}',\pr\otimes\pr')$ of the original probability space), then Corollary \ref{modified delta method for the bootstrap - II}, Example \ref{example iid data}, Examples \ref{Example DR Functional}--\ref{Example HD of DR Functional}, and Theorem \ref{bootstrap results of VdV-W} show that $(f_g(\widehat F_n^*))$ is a bootstrap version in probability of $(f_g(\widehat F_n))$. This bootstrap consistency can also be obtained by results on L-statistics by \citet{Helmersetal1990} and \citet{Gribkova2002}. However, the latter results rely on the independence of $X_1,X_2,\ldots$. To the best of our knowledge so far there do not exit general results on bootstrap consistency for empirical distortion risk measures associated with continuous concave distortion functions when the data $X_1,X_2,\ldots$ are dependent. On the other hand, our theory admits such results. Indeed, if the sequence $(X_i)$ is $\beta$-mixing with mixing rate as in condition (b) of Theorem \ref{bootstrap results of Radulovic}, $\int\phi^p\,dF<\infty$ for some $p>2$, and $\widehat F_n^*$ is as in Theorem \ref{bootstrap results of Radulovic}, then Corollary \ref{modified delta method for the bootstrap - II}, Example \ref{example weakly dependent data}, Examples \ref{Example DR Functional}--\ref{Example HD of DR Functional}, and Theorem \ref{bootstrap results of Radulovic} show that $(f_g(\widehat F_n^*))$ is a bootstrap version in probability of $(f_g(\widehat F_n))$. We emphasize that the results by \citet[Chapter 4.4]{Lahiri2003} for $\alpha$-mixing data do not cover this bootstrap consistency, because Lahiri assumes Fr\'{e}chet differentiable for $f_g$ which fails for continuous concave distortion functions $g$.

2) Let $f_h:\bD(f_h)\rightarrow\R$ be the V-functional corresponding to a given Borel measurable function $h:\R^2\rightarrow\R$ (sometimes referred to as kernel) which is given by
\begin{equation}\label{def of U}
    f_h(F):=\iint h(x_1,x_2)\,dF(x_1)dF(x_2),
\end{equation}
where $\bD(f_h)$ denotes the set of all distribution functions on the real line for which the double integral in (\ref{def of U}) exists. It was shown in  Theorem 4.1 in \cite{BeutnerZaehle2012} that subject to some regularity conditions on $h$ and $F$ the V-functional $f_h$ is quasi-Hadamard differentiable at $F$ w.r.t.\ a suitable nonuniform sup-norm. Similar as in Example \ref{Example HD of DR Functional} it can be shown that condition (d) of Corollary \ref{modified delta method for the bootstrap - II} holds for the quasi-Hadamard derivative of $f_h$. Then again, if $(X_i)$ is a stationary $\beta$-mixing sequence of random variables with distribution function $F$ and mixing rate as in condition (b) of Theorem \ref{bootstrap results of Radulovic}, $\int\phi^p\,dF<\infty$ for some $p>2$, and $\widehat F_n^*$ is as in Theorem \ref{bootstrap results of Radulovic}, Corollary \ref{modified delta method for the bootstrap - II} shows that $(f_h(\widehat F_n^*))$ is a bootstrap version in probability of $(f_h(\widehat F_n))$. Other approaches to show bootstrap consistency for non-degenerate U- and V-statistics can be found, for example, in \citet{ArconesGine1992}, \citet{Janssen1994}, and \citet{DehlingWendler2010} (yet another approach was exemplified for the variance by \citet{Dudley1990}); see also \citet{BuecherKojadinovic2015} who use results of \citet{DehlingWendler2010}. Among other things \citet[Theorem 2.1]{DehlingWendler2010} also establish bootstrap consistency for non-degenerate U- and V-statistics for $\beta$-mixing sequences. Whereas their approach requires an additional integrability condition on $(X_1,X_k)$, our approach (based on Corollary \ref{modified delta method for the bootstrap - II} that we just outlined) requires stronger regularity conditions on the kernel $h$. Looking at condition (b) in Theorem \ref{bootstrap results of Radulovic} and the condition on the mixing coefficient in \citet[Theorem 2.1]{DehlingWendler2010}, it seems that both approaches impose the same condition on the mixing coefficient. Thus, the approach based on Corollary \ref{modified delta method for the bootstrap - II} may supplement the results in \cite{DehlingWendler2010}.


\section{Proof of Theorem \ref{modified delta method for the bootstrap}}\label{Proof of Abstract Delta Method}

We start with a convention and a general remark. We will equip the product space $\overline{\bE}:=\bE\times\bE$ with the metric $\overline{d}((x_1,x_2),(y_1,y_2)):=\max\{\|x_1-y_1\|_{\bE};\|x_2-y_2\|_{\bE}\}$, and denote the corresponding open-ball $\sigma$-algebra on $\overline\bE$ by $\overline{\cal B}^\circ$. Note that $\overline{\cal B}^\circ\subseteq{\cal B}^\circ\otimes{\cal B}^\circ$, because any $\overline d$-open ball in $\overline\bE$ is the product of two $\|\cdot\|_\bE$-open balls in $\bE$. Analogously the product space $\overline{\widetilde{\bE}}:=\widetilde\bE\times\widetilde\bE$ will be equipped with the metric $\overline{\widetilde d}((\widetilde x_1,\widetilde x_2),(\widetilde y_1,\widetilde y_2)):=\max\{\|\widetilde x_1-\widetilde y_1\|_{\widetilde\bE};\|\widetilde x_2-\widetilde y_2\|_{\widetilde\bE}\}$. By the separability of $(\widetilde\bE,\|\cdot\|_{\widetilde\bE})$ the corresponding Borel $\sigma$-algebra $\overline{\widetilde{\cal B}}$ coincides with the product $\sigma$-algebra $\widetilde{\cal B}\otimes\widetilde{\cal B}$; cf.\ \citet[Proposition 4.1.7]{Dudley2002}. So the couple $(\xi_1,\xi_2)$ is an $(\overline{\widetilde\bE},\overline{\widetilde{\cal B}})$-valued random variable when $\xi_1$ and $\xi_2$ are $(\widetilde\bE,\widetilde{\cal B})$-valued random variables. In particular, $h(\xi_1,\xi_2)$ is an $(\widetilde\bE,\widetilde{\cal B})$-valued random variable when $h:\overline{\widetilde{\bE}}\to\widetilde\bE$ is continuous. Since the addition and the multiplication by constants in normed vector spaces are continuous, we have in particular that a linear combination of two $(\widetilde\bE,\widetilde{\cal B})$-valued random variables is again an $(\widetilde\bE,\widetilde{\cal B})$-valued random variable. This fact will be used frequently in the sequel without further mentioning.

(i): By assumption (b) we have that $f(\widehat T_n)$ is $({\cal F},\widetilde{\cal B})$-measurable. This implies that $a_n(f(\widehat T_n)-f(\theta))$ is $({\cal F},\widetilde{\cal B})$-measurable for every $n\in\N$, because we assumed that $(\widetilde\bE,\|\cdot\|_{\widetilde\bE})$ is separable. Now, assertion (i) directly follows from the functional delta-method in the form of Theorem \ref{modified delta method}.

(ii): Recall that $\widehat T_n$ will frequently be seen as a map defined on the extension $\overline\Omega$ of $\Omega$. From the above we therefore have that $f(\widehat T_n)$ is $(\overline{\cal F},\widetilde{\cal B})$-measurable. Moreover, $f(\widehat T_n^*)$ is $(\overline{\cal F},\widetilde{\cal B})$-measurable due assumption (e). In particular, the map $a_n(f(\widehat T_n^*)-f(\widehat T_n))$ is $(\overline{\cal F},\widetilde{\cal B})$-measurable, because we assumed that $(\widetilde\bE,\|\cdot\|_{\widetilde\bE})$ is separable. By assumptions (a) and (d) we also have that the map $\dot f_\theta(\xi)$ is $(\mathcal{F}_{0},\widetilde{\cal B})$-measurable, and by assumptions (d) and (f) we have that the map $\dot f_\theta(a_n(\widehat T_n^*-\widehat T_n))$ is $(\overline{\cal F},\widetilde{\cal B})$-measurable.

To verify (\ref{modified delta method for the bootstrap - assumption - 40}), we will adapt the arguments of Section 3.9.3 in \citet{van der Vaart Wellner 1996}. First note that $\widetilde Q_n$ defined by
$$
    \widetilde Q_n(\omega,\widetilde A)\,:=\,\pr'\circ\big\{\dot f_\theta\big(a_n(\widehat T_n^*(\omega,\cdot)-\widehat T_n(\omega))\big)\big\}^{-1}[\widetilde A],\qquad\omega\in\Omega,\,\widetilde A\in\widetilde{\cal B}
$$
provides a conditional distribution of $\dot f_\theta(a_n(\widehat T_n^*-\widehat T_n))$ given $\Pi$. This follows from Lemma \ref{reresentation of P n} (with $X(\omega,\omega')=g(\omega,\omega')=\dot f_\theta(a_n(\widehat T_n^*(\omega,\omega')-\widehat T_n(\omega)))$ and $Y=\Pi$). Now, let $\delta>0$ be arbitrary but fixed. For (\ref{modified delta method for the bootstrap - assumption - 40}) it suffices to show that
\begin{equation}\label{proof of modified delta method for the bootstrap - 20}
    \lim_{n\to\infty}\,\pr\Big[\Big\{\omega\in\Omega:\,\widetilde{d}_{\scriptsize{\rm BL}}\big(\widetilde P_n(\omega,\cdot),\widetilde Q_n(\omega,\cdot)\big)\ge\frac{\delta}{2}\Big\}\Big]=\,0
\end{equation}
and
\begin{equation}\label{proof of modified delta method for the bootstrap - 30}
    \lim_{n\to\infty}\,\pr\Big[\Big\{\omega\in\Omega:\,\widetilde d_{\scriptsize{\rm BL}}\big(\widetilde Q_n(\omega,\cdot),\mbox{\rm law}\{\dot f_\theta(\xi)\}\big)\ge\frac{\delta}{2}\Big\}\Big]=\,0.
\end{equation}
Note that the maps $\omega\mapsto \widetilde d_{\scriptsize{\rm BL}}(\widetilde P_n(\omega,\cdot),\widetilde Q_n(\omega,\cdot))$ and $\omega\mapsto \widetilde d_{\scriptsize{\rm BL}}(\widetilde Q_n(\omega,\cdot),\mbox{\rm law}\{\dot f_\theta(\xi)\})$ are $({\cal F},{\cal B}(\R_+))$-measurable, because $(\widetilde\bE,\|\cdot\|_{\widetilde\bE})$ was assumed to be separable. For the latter map one can argue as subsequent to Definition \ref{abstract bootstrap - definition - in outer probability}. For the former map one can argue in the same way, noting that $(\widetilde{\cal M}_1,\widetilde d_{\scriptsize{\rm BL}})$ is separable (cf.\ Remark \ref{remark appendix b separable implies metric} and Theorem \ref{aliprantis}) and that the metric distance of two random variables in a separable metric space is also measurable (cf.\ \citet[Lemma 6.1]{Klenke2014}). In particular, the events in (\ref{proof of modified delta method for the bootstrap - 20}) and (\ref{proof of modified delta method for the bootstrap - 30}) are ${\cal F}$-measurable.

We first show (\ref{proof of modified delta method for the bootstrap - 30}). By (\ref{modified delta method for the bootstrap - assumption - 15}) in assumption (f), the Continuous Mapping theorem in the form of \citet[Theorem 6.4]{Billingsley1999} (along with $\pr_0\circ\xi^{-1}[\bE_0]=1$ and the continuity of $\dot f_\theta$), and the implication (a)$\Rightarrow$(g) in the Portmanteau theorem \ref{Portemanteau}, we have
$$
    \lim_{n\to\infty}\widetilde d_{\scriptsize{\rm BL}}\big(\widetilde Q_n(\omega,\cdot),\mbox{\rm law}\{\dot f_\theta(\xi)\}\big)\,=\,0\qquad\mbox{$\pr$-a.e.\ $\omega$}.
$$
Since almost sure convergence of real-valued random variables implies convergence in probability, we arrive at (\ref{proof of modified delta method for the bootstrap - 30}).

To verify (\ref{proof of modified delta method for the bootstrap - 20}), we set
$$
    \eta_{n}(\omega,\omega')\, := \, a_n\big(f(\widehat T_n^*(\omega,\omega'))-f(\widehat T_n(\omega))\big)-\dot f_\theta\big(a_n(\widehat T_n^*(\omega,\omega')-\widehat T_n(\omega))\big)
$$
and
$$
    \eta_{n,\widetilde h}(\omega,\omega') \,  := \, \widetilde h\Big(a_n\big(f(\widehat T_n^*(\omega,\omega'))-f(\widehat T_n(\omega))\big)\Big)-\widetilde h\Big(\dot f_\theta\big(a_n(\widehat T_n^*(\omega,\omega')-\widehat T_n(\omega))\big)\Big)
$$
for every $\widetilde h\in\widetilde{\rm BL}_1$ with $\widetilde{\rm BL}_1$ as defined before (\ref{def bl metric - add on}). We then obtain
\begin{eqnarray}
    \lefteqn{\pr\Big[\Big\{\omega\in\Omega:\,\widetilde d_{\scriptsize{\rm BL}}\big(\widetilde P_n(\omega,\cdot),\widetilde Q_n(\omega,\cdot)\big)\ge\frac{\delta}{2}\Big\}\Big]}\nonumber\\
    & = & \pr\Big[\Big\{\omega\in\Omega:\,\sup_{\widetilde h\in\widetilde{\rm BL}_1}\Big|\int\widetilde h(\widetilde x)\,\widetilde P_n(\omega,d\widetilde x)-\int\widetilde h(\widetilde x)\,\widetilde Q_n(\omega,d\widetilde x)\Big|\ge\frac{\delta}{2}\Big\}\Big]\nonumber\\
    & = & \pr\Big[\Big\{\omega\in\Omega:\,\sup_{\widetilde h\in\widetilde{\rm BL}_1}\Big|\int\widetilde h\Big(a_n\big(f(\widehat T_n^*(\omega,\omega'))-f(\widehat T_n(\omega))\big)\Big)\,\pr'[d\omega']\nonumber\\
    & & \qquad\qquad\qquad\qquad-\,\int\widetilde h\Big(\dot f_\theta\big(a_n(\widehat T_n^*(\omega,\omega')-\widehat
    T_n(\omega))\big)\Big)\,\pr'[d\omega']\Big|\,\ge\,\frac{\delta}{2}\Big\}\Big]\nonumber\\
    & = & \pr\Big[\Big\{\omega\in\Omega:\,\sup_{\widetilde h\in\widetilde{\rm BL}_1}\Big|\int \eta_{n,\widetilde h}(\omega,\omega')\,\pr'[d\omega']\Big|\,\ge\,\frac{\delta}{2}\Big\}\Big]\nonumber\\
    & = & \pr^{\sf out}\Big[\Big\{\omega\in\Omega:\,\sup_{\widetilde h\in\widetilde{\rm BL}_1}\Big|\int \eta_{n,\widetilde h}(\omega,\omega')\,\pr'[d\omega']\Big|\,\ge\,\frac{\delta}{2}\Big\}\Big]\nonumber\\
    & \le & \pr^{\sf out}\Big[\Big\{\omega\in\Omega:\,\sup_{\widetilde h\in\widetilde{\rm BL}_1}\int \big| \eta_{n,\widetilde h}(\omega,\omega')\big|\,\pr'[d\omega']\,\ge\,\frac{\delta}{2}\Big\}\Big]\nonumber\\
    & = & \pr^{\sf out}\Big[\Big\{\omega\in\Omega:\,\sup_{\widetilde h\in\widetilde{\rm BL}_1}\Big(\int \big| \eta_{n,\widetilde h}(\omega,\omega')\big|\,\eins_{\{|\eta_{n,\widetilde h}|<\delta/4\}}(\omega,\omega')\,\pr'[d\omega']\nonumber\\
    & & \qquad\qquad\qquad\qquad+\int \big| \eta_{n,\widetilde h}(\omega,\omega')\big|\,\eins_{\{|\eta_{n,\widetilde h}|\ge\delta/4\}}(\omega,\omega')\,\pr'[d\omega']\Big)\,\ge\,\frac{\delta}{2}\Big\}\Big]\nonumber\\
    & \le & \pr^{\sf out}\Big[\Big\{\omega\in\Omega:\,\frac{\delta}{4}+\sup_{\widetilde h\in\widetilde{\rm BL}_1}\int \big| \eta_{n,\widetilde h}(\omega,\omega')\big|\,\eins_{\{|\eta_{n,\widetilde h}|\ge\delta/4\}}(\omega,\omega')\,\pr'[d\omega']\,\ge\,\frac{\delta}{2}\Big\}\Big]\nonumber\\
    & \le & \pr^{\sf out}\Big[\Big\{\omega\in\Omega:\,\sup_{\widetilde h\in\widetilde{\rm BL}_1}\int 2\,\eins_{\{|\eta_{n,\widetilde h}|\ge\delta/4\}}(\omega,\omega')\,\pr'[d\omega']\,\ge\,\frac{\delta}{4}\Big\}\Big]\nonumber\\
    & \le & \pr^{\sf out}\Big[\Big\{\omega\in\Omega:\,2\int \eins_{\{\|\eta_{n}\|_{\widetilde\bE}\ge\delta/4\}}(\omega,\omega')\,\pr[d\omega']\,\ge\,\frac{\delta}{4}\Big\}\Big],\label{proof of modified delta method for the bootstrap - 33}
\end{eqnarray}
where the second last and the last step are ensured by $\|\widetilde h\|_\infty\le 1$ and the Lipschitz continuity of $h$ (with Lipschitz constant $1$), respectively. We have seen above that the maps $a_n(f(\widehat T_n)-f(\widehat T_n^*))$ and $\dot f_\theta(a_n(\widehat T_n^*-\widehat T_n))$ are $(\overline{\cal F},\widetilde{\cal B})$-measurable. Since $(\widetilde\bE,\|\cdot\|_{\widetilde\bE})$ is separable, we can conclude that the map $\eta_n$ is $(\overline{\cal F},\widetilde{\cal B})$-measurable. Since the map $\|\cdot\|_{\widetilde\bE}:\widetilde\bE\rightarrow\R_+$ is continuous and thus $(\widetilde{\cal B},{\cal B}(\R_+))$-measurable, we obtain that the map $\|\eta_n\|_{\widetilde\bE}:\overline{\Omega}\rightarrow\R_+$ is $(\overline{\cal F},{\cal B}(\R_+))$-measurable. By Fubini's theorem we can conclude that the map
$$
    \omega\,\longmapsto\,\int \eins_{\{\|\eta_{n}\|_{\widetilde\bE}\ge\delta/4\}}(\omega,\omega')\,\pr'[d\omega']
$$
is $({\cal F},{\cal B}(\R_+))$-measurable. Therefore, we may replace the outer probability $\pr^{\sf out}$ by the ordinary probability $\pr$ in the last line of (\ref{proof of modified delta method for the bootstrap - 33}). So we obtain
\begin{eqnarray}
    \lefteqn{\pr\Big[\Big\{\omega\in\Omega:\,\widetilde d_{\scriptsize{\rm BL}}\big(\widetilde P_n(\omega,\cdot),\widetilde Q_n(\omega,\cdot)\big)\ge\frac{\delta}{2}\Big\}\Big]}\nonumber\\
    & \le & \pr\Big[\Big\{\omega\in\Omega:\,2\int \eins_{\{\|\eta_{n}\|_{\widetilde\bE}\ge\delta/4\}}(\omega,\omega')\,\pr'[d\omega']\,\ge\,\frac{\delta}{4}\Big\}\Big]\nonumber\\
    & = & \pr\Big[\Big\{\omega\in\Omega:\,\pr'\Big[\Big\{\omega'\in\Omega':\,\|\eta_{n}(\omega,\omega')\|_{\widetilde\bE}\ge\frac{\delta}{4}\Big\}\Big]\ge\,\frac{\delta}{8}\Big\}\Big]\nonumber\\
    & \le & \frac{8}{\delta}\int\pr'\Big[\Big\{\omega'\in\Omega':\,\|\eta_{n}(\omega,\omega')\|_{\widetilde\bE}\ge\frac{\delta}{4}\Big\}\Big]\,\pr[d\omega]\nonumber\\
    & = & \frac{8}{\delta}\,\overline\pr\Big[\Big\{(\omega,\omega')\in\overline\Omega:\,\|\eta_{n}(\omega,\omega')\|_{\widetilde\bE}\ge\frac{\delta}{4}\Big\}\Big]\nonumber\\
    & = & \frac{8}{\delta}\,\overline\pr\Big[\big\|a_n\big(f(\widehat T_n^*)-f(\widehat T_n)\big)-\,\dot f_\theta\big(a_n(\widehat T_n^*-\widehat T_n)\big)\big\|_{\widetilde\bE}\ge\frac{\delta}{4}\Big],\nonumber
\end{eqnarray}
where for the third and the fourth step we used respectively Markov's inequality and the representation of the product measure $\overline\pr=\pr\otimes\pr'$ as given in \citet[Formula (23.3)]{Bauer2001}. Thus, it remains to show that
\begin{equation}\label{proof of modified delta method for the bootstrap - 35}
    a_n\big(f(\widehat T_n^*)-f(\widehat T_n)\big)-\dot f_\theta\big(a_n(\widehat T_n^*-\widehat T_n)\big)\,\rightarrow^{\sf p}\,0_{\widetilde\bE}
    \qquad\mbox{in $(\widetilde\bE,\|\cdot\|_{\widetilde\bE})$ w.r.t.\ }\overline{\pr},
\end{equation}
where $\rightarrow^{\sf p}$ refers to convergence in probability and  $0_{\widetilde\bE}$ denotes the null in $\widetilde\bE$. To prove (\ref{proof of modified delta method for the bootstrap - 35}), we note that by assumption (b) we have that $a_n(\widehat T_n-\theta)$ converges in distribution$^\circ$ to some separable random variable, $\xi$. So we may apply part (ii) of Theorem \ref{modified delta method} to obtain
\begin{equation}\label{proof of modified delta method for the bootstrap - 40}
    a_n\big(f(\widehat T_n)-f(\theta)\big)-\dot f_\theta\big(a_n(\widehat T_n-\theta)\big)\,\rightarrow^{\sf p}\,0_{\widetilde\bE}
    \qquad\mbox{in $(\widetilde\bE,\|\cdot\|_{\widetilde\bE})$ w.r.t.\ }\overline{\pr},
\end{equation}
where condition (g) of Theorem \ref{modified delta method} holds since $(\widetilde\bE,\|\cdot\|_{\widetilde\bE})$ was assumed to be separable (cf.\ the discussion at the beginning of the proof).
Further, in the following we will show that $a_n(\widehat T_n^*-\theta)$ converges in distribution$^\circ$ to some separable random variable too. So we may apply part (ii) of Theorem \ref{modified delta method} once more to obtain
\begin{equation}\label{proof of modified delta method for the bootstrap - 50}
    a_n\big(f(\widehat T_n^*)-f(\theta)\big)-\dot f_\theta\big(a_n(\widehat T_n^*-\theta)\big)\,\rightarrow^{\sf p}\,0_{\widetilde\bE}
    \qquad\mbox{in $(\widetilde\bE,\|\cdot\|_{\widetilde\bE})$ w.r.t.\ }\overline{\pr}.
\end{equation}
Now, (\ref{proof of modified delta method for the bootstrap - 40}), (\ref{proof of modified delta method for the bootstrap - 50}) and the linearity of $\dot f_\theta$ imply (\ref{proof of modified delta method for the bootstrap - 35}). 

It remains to show that $a_n(\widehat T_n^*-\theta)$ converges in distribution$^\circ$ to some separable random variable. For this it suffices to show that $(a_n(\widehat T_n-\theta),a_n(\widehat T_n^*-\widehat T_n))$ converges in distribution$^\circ$ to $(\xi_1,\xi_2)$ in $(\overline{\bE},\overline{\cal B}^\circ,\overline{d})$, where $(\xi_1,\xi_2)$ is an $(\overline{\bE},\overline{\cal B}^\circ)$-valued random variable (on some probability space) which takes values only in $\overline{\bE}_0:=\bE_0\times\bE_0$. In fact, the extended Continuous Mapping theorem \ref{Skorohod - CMT} applied to the functions $h_n:\overline{\bE}\rightarrow\bE$ and $h_0:\overline{\bE}_0\rightarrow\bE_0\subseteq\bE$ given by respectively $h_n(x,y):=x+y$ and $h_0(x,y):=x+y$ then implies that $a_n(\widehat T_n^*-\theta)=a_n(\widehat T_n^*-\widehat T_n)+a_n(\widehat T_n-\theta)$ converges in distribution$^\circ$ to the separable random variable $\xi_1+\xi_2$. For the application of the extended Continuous Mapping theorem note that $h_n(a_n(\widehat T_n^*-\widehat T_n),a_n(\widehat T_n-\theta))=a_n(\widehat T_n^*-\theta)$ is $(\overline{\cal F},{\cal B}^\circ)$-measurable by the first part of condition (g) and that the map $h_0:\overline{\bE}_0\rightarrow\bE$ is continuous and $(\overline{\cal B}_0,{\cal B}^\circ)$-measurable for $\overline{\cal B}_0:=\overline{\cal B}^\circ\cap\overline{\bE}_0=\overline{\cal B}\cap\overline{\bE}_0$. For the latter measurability take into account that $\overline{\bE}_0$ is separable w.r.t.\ $\overline{d}$ and argue as at the beginning of the proof. Also note that $(a_n(\widehat T_n-\theta),a_n(\widehat T_n^*-\widehat T_n))$ can be seen as an $(\overline{\bE},\overline{\cal B}^\circ)$-valued random variable, because it is obviously $({\cal F},{\cal B}^\circ\otimes{\cal B}^\circ)$-measurable and $\overline{\cal B}^\circ\subseteq{\cal B}^\circ\otimes{\cal B}^\circ$.

To show that $(a_n(\widehat T_n-\theta),a_n(\widehat T_n^*-\widehat T_n))$ converges in distribution$^\circ$ to some separable random element $(\xi_1,\xi_2)$, we will adapt some of the arguments of the proof of Theorem 2.2 in \citet{Kosorok2008} where weak convergence is understood in the Hoffmann-J{\o}rgensen sense. Let $(\Omega_1\times\Omega_2,{\cal F}_1\otimes{\cal F}_2,\pr_1\otimes\pr_2):=(\overline{\bE},{\cal B}^\circ\otimes{\cal B}^\circ,(\pr_0\circ\xi^{-1})\otimes(\pr_0\circ\xi^{-1}))$ (with $\xi$ and $\pr_0$ as in condition (b)) and $\xi_i$ be the $i$-th coordinate projection on $\Omega_1\times\Omega_2=\overline{\bE}$, $i=1,2$. Then $(\xi_1,\xi_2)$ can be seen as an $(\overline{\bE},\overline{\cal B}^\circ)$-valued random variable on $(\Omega_1\times\Omega_2,{\cal F}_1\otimes{\cal F}_2,\pr_1\otimes\pr_2)$, because by $\overline{\cal B}^\circ\subseteq{\cal B}^\circ\otimes{\cal B}^\circ$ it is clearly $({\cal B}^\circ\otimes{\cal B}^\circ,\overline{\cal B}^\circ)$-measurable. In view of the implication (f)$\Rightarrow$(a) in the Portmanteau theorem \ref{Portemanteau}, for the convergence in distribution$^\circ$ of the pair $(a_n(\widehat T_n-\theta),a_n(\widehat T_n^*-\widehat T_n))$ to the random variable $(\xi_1,\xi_2)$ it suffices to show that
$$
    \int \overline{h}\big(a_n(\widehat T_n-\theta),a_n(\widehat T_n^*-\widehat T_n)\big)\,d(\pr\otimes\pr')\,\longrightarrow\,\int \overline{h}(\xi_1,\xi_2)\,d(\pr_1\otimes\pr_2)
$$
for every $\overline{h}\in\overline{\rm BL}_1^\circ$, where $\overline{\rm BL}_1^\circ$ denotes the set of all real-valued functions on $\overline\bE=\bE\times\bE$ that are $(\overline{\cal B}^\circ,{\cal B}(\R))$-measurable, bounded by $1$ and Lipschitz continuous with Lipschitz constant $1$ (as defined before (\ref{def bl metric - add on})). So, let $\overline{h}\in\overline{{\rm BL}}_1^\circ$. We have
\begin{eqnarray*}
    \lefteqn{\Big|\int\overline{h}(a_n(\widehat T_n-\theta),a_n(\widehat T_n^*-\widehat T_n))\,d(\pr\otimes\pr')-\int\overline{h}(\xi_1,\xi_2)\,d(\pr_1\otimes\pr_2)\Big|}\\
    & \le & \Big|\int\overline{h}(a_n(\widehat T_n-\theta),a_n(\widehat T_n^*-\widehat T_n))\,d(\pr\otimes\pr')-\int\overline{h}(a_n(\widehat T_n-\theta),\xi_2)\,d(\pr\otimes\pr_2)\Big|\\
    & & +\,\Big|\int\overline{h}(a_n(\widehat T_n-\theta),\xi_2)\,d(\pr\otimes\pr_2)-\int\overline{h}(\xi_1,\xi_2)\,d(\pr_1\otimes\pr_2)\big]\Big|\\
    & =: & S_1(n)\,+\,S_2(n).
\end{eqnarray*}
For every $x_2\in\bE$, define the function $h_{x_2}:\bE\rightarrow\R$ by $h_{x_2}(x_1):=\overline{h}(x_1,x_2)$ and note that $h_{x_2}$ is bounded, continuous, and $({\cal B}^{\circ},{\cal B}(\R))$-measurable. The latter measurability means that $h_{x_2}^{-1}(B)=(\overline{h}^{-1}(B))_{x_2}:=\{x_1\in\bE:(x_1,x_2)\in\overline{h}^{-1}(B)\}$ lies in ${\cal B}^\circ$ for every $B\in{\cal B}(\R)$. By the $(\overline{{\cal B}}^\circ,{\cal B}(\R))$-measurability of $\overline{h}$ the set $\overline{h}^{-1}(B)$ lies in $\overline{{\cal B}}^\circ$. By Lemma 23.1 of \cite{Bauer2001} the set $A_{x_2}:=\{x_1\in\bE:(x_1,x_2)\in A\}$ lies in ${\cal B}^\circ$ for every $A\in{\cal B}^\circ\otimes{\cal B}^\circ$. Thus, in view of $\overline{\cal B}^\circ\subseteq{\cal B}^\circ\otimes{\cal B}^\circ$, it follows that the functions $h_{x_2}$ is indeed $({\cal B}^{\circ},{\cal B}(\R))$-measurable. Now, with the help of Fubini's theorem we obtain
\begin{eqnarray*}
    \lefteqn{S_2(n)}\\
    & \le & \int\Big|\int\overline{h}(a_n(\widehat T_n(\omega)-\theta),\xi_2(\omega_2))\,\pr[d\omega]-\int \overline{h}(\xi_1(\omega_1),\xi_2(\omega_2))\,\pr_1(d\omega_1)\Big|\,\pr_2[d\omega_2]\\
    & = & \int\Big|\int h_{\xi_2(\omega_2)}(a_n(\widehat T_n(\omega)-\theta))\,\pr[d\omega]-\int h_{\xi_2(\omega_2)}(\xi_1(\omega_1))\,\pr_1(d\omega_1)\Big|\,\pr_2[d\omega_2].
\end{eqnarray*}
In view of assumption (a), the integrand of the outer integral converges to $0$ for every $\omega_2$. So, since 
$\|h_{x_2}(\cdot)\|_\infty\le 1$ for every $x_2\in\bE$, the Dominated Convergence theorem implies that the summand $S_2(n)$ converges to $0$.
For every $x_1\in\bE$, define the function $h_{x_1}:\bE\rightarrow\R$ by $h_{x_1}(x_2):=\overline{h}(x_1,x_2)$ and note that $h_{x_1}\in{\rm BL}_1^\circ$ for every $x_1\in\bE$ (for the measurability of $h_{x_1}$ one can argue as for $h_{x_2}$ above). With the help of Fubini's theorem we obtain
\begin{eqnarray}\label{eq proof main thm before (iii)}
    S_1(n)
    & \le & \int\Big|\int\overline{h}(a_n(\widehat T_n(\omega)-\theta),a_n(\widehat T_n^*(\omega,\omega')-\widehat T_n(\omega)))\,\pr'[d\omega'] \nonumber \\
    & & \qquad\qquad\qquad-\int\overline{h}(a_n(\widehat T_n(\omega)-\theta),\xi_2(\omega_2))\,\pr_2[d\omega_2]\Big|\,\pr[d\omega]\nonumber \\
    & = & \int\Big|\int h_{a_n(\widehat T_n(\omega)-\theta)}(a_n(\widehat T_n^*(\omega,\omega')-\widehat T_n(\omega)))\,\pr'[d\omega'] \nonumber \\
    & & \qquad\qquad\qquad-\int h_{a_n(\widehat T_n(\omega)-\theta)}(\xi_2(\omega_2))\,\pr_2[d\omega_2]\Big|\,\pr[d\omega] \nonumber \\
    & \le & \int\sup_{m\in\N}\Big|\int h_{a_m(\widehat T_m(\omega)-\theta)}(a_n(\widehat T_n^*(\omega,\omega')-\widehat T_n(\omega)))\,\pr'[d\omega'] \nonumber\\
    & & \qquad\qquad\qquad-\int h_{a_m(\widehat T_m(\omega)-\theta)}(\xi_2(\omega_2))\,\pr_2[d\omega_2]\Big|\,\pr[d\omega].
\end{eqnarray}
The integrand of the outer integral is bounded above by $d_{\scriptsize{\rm BL}}^\circ(P_n(\omega,\cdot),\mbox{\rm law}\{\xi_2\})$. So it follows by the second part of assumption (f) and the implication (a)$\Rightarrow$(g) in the Portmanteau theorem \ref{Portemanteau} that the integrand of the outer integral converges to $0$ for $\pr$-a.e.\ $\omega$. In view of  $\|h_{a_m(\widehat T_m(\omega)-\theta)}(\cdot)\|_\infty\le 1$ for every $m\in\N$, the Dominated Convergence theorem implies that the summand $S_1(n)$ converges to $0$ too. This completes the proof of part (ii).

(iii): One can proceed as for the proof of part (ii). It again suffices to show (\ref{proof of modified delta method for the bootstrap - 20}) and (\ref{proof of modified delta method for the bootstrap - 30}). The proof of (\ref{proof of modified delta method for the bootstrap - 20}) can be transferred nearly verbatim. The convergence of the upper bound in (\ref{eq proof main thm before (iii)}) to zero was justified by the classical Dominated Convergence theorem. This time one has to use slightly different arguments. The upper bound in (\ref{eq proof main thm before (iii)}) is bounded above by
$$
    \int^{\scriptsize{\sf out}}\sup_{h\in {\rm BL}_1^\circ}\Big|\int h(a_n(\widehat T_n^*(\omega,\omega')-\widehat T_n(\omega)))\,\pr'[d\omega']-\int h(\xi_2(\omega_2))\,\pr_2[d\omega_2]\Big|\,\pr[d\omega],
$$
which equals
$$
    \int^{\scriptsize{\sf out}} d_{{\rm BL}}^\circ(P_n(\omega,\cdot),{\rm law}\{\xi\})\,\pr[d\omega].
$$
Here $\int^{\scriptsize{\sf out}}$ refers to the outer integral (outer expectation). By (\ref{modified delta method for the bootstrap - assumption - 18}) in assumption (f'), the integrand of the latter integral converges to $0$ in outer probability. Lemma 3.3.4 in \cite{Dudley1999} then implies
$$
    \limsup_{n\to\infty}\int^{\scriptsize{\sf out}} d_{{\rm BL}}^\circ(P_n(\omega,\cdot),{\rm law}\{\xi\})\,\pr[d\omega]\,\le\,0.
$$
It follows that the summand $S_1(n)$ again converges to $0$. This gives (\ref{proof of modified delta method for the bootstrap - 20}).

It remains to show that (\ref{proof of modified delta method for the bootstrap - 30}) can also be derived from assumption (f'). We have
\begin{eqnarray}
    \lefteqn{\pr\Big[\Big\{\omega\in\Omega:\,\widetilde d_{\scriptsize{\rm BL}}\big(\widetilde Q_n(\omega,\cdot),\mbox{\rm law}\{\dot f_\theta(\xi)\}\big)\ge\frac{\delta}{2}\Big\}\Big]}\nonumber\\
    & = & \pr\Big[\Big\{\omega\in\Omega:\,\sup_{\widetilde h\in\widetilde{\rm BL}_1}\Big|\int\widetilde h(\widetilde x)\,\widetilde Q_n(\omega,d\widetilde x)-\int\widetilde h(\widetilde x)\,\mbox{\rm law}\{\dot f_\theta(\xi)\}[d\widetilde x]\Big|\ge\frac{\delta}{2}\Big\}\Big]\nonumber\\
    & = & \pr\Big[\Big\{\omega\in\Omega:\,\sup_{\widetilde h\in\widetilde{\rm BL}_1}\Big|\int\widetilde h\Big(\dot f_\theta\big(a_n(\widehat T_n^*(\omega,\omega')-\widehat
    T_n(\omega))\big)\Big)\,\pr'[d\omega']\nonumber\\
    & & \qquad\qquad\qquad\qquad-\,\int\widetilde h\big(\dot f_\theta(\xi(\omega_0))\big)\,\pr_0[d\omega_0]\Big|\ge\,\frac{\delta}{2}\Big\}\Big]\nonumber\\
    & = & \pr\Big[\Big\{\omega\in\Omega:\,\sup_{\widetilde h\in\widetilde{\rm BL}_1}\Big|\int\widetilde h\circ\dot f_\theta\big(a_n(\widehat T_n^*(\omega,\omega')-\widehat
    T_n(\omega))\big)\,\pr'[d\omega']\nonumber\\
    & & \qquad\qquad\qquad\qquad-\,\int\widetilde h\circ\dot f_\theta\big(\xi(\omega_0)\big)\,\pr_0[d\omega_0]\Big|\ge\,\frac{\delta}{2}\Big\}\Big]\nonumber\\
    & = & \pr\Big[\Big\{\omega\in\Omega:\,\sup_{\widetilde h\in\widetilde{\rm BL}_1}\Big|\int\widetilde h\circ\dot f_\theta(x)\,P_n(\omega,dx)\nonumber\\
    & & \qquad\qquad\qquad\qquad-\,\int\widetilde h\circ\dot f_\theta(x)\,\mbox{\rm law}\{\xi\}[dx]\Big|\ge\,\frac{\delta}{2}\Big\}\Big]\nonumber\\
    & \le & \pr^{\sf out}\Big[\Big\{\omega\in\Omega:\,\sup_{h\in{\rm BL}_1^\circ}\Big|\int h(x)\,P_n(\omega,dx)\,-\,\int h(x)\,\mbox{\rm law}\{\xi\}[dx]\Big|\ge\,\frac{\delta}{2(L_{f,\theta}\vee 1)}\Big\}\Big]\nonumber\\
    & = & \pr^{\sf out}\Big[\Big\{\omega\in\Omega:\,d_{\scriptsize{\rm BL}}^\circ\big(P_n(\omega,\cdot),\mbox{\rm law}\{\xi\}\big)\ge\frac{\delta}{2(L_{f,\theta}\vee 1)}\Big\}\Big],\label{proof of modified delta method for the bootstrap - 60}
\end{eqnarray}
where $L_{f,\theta}>0$ denotes the Lipschitz constant of the linear and continuous (thus Lipschitz continuous) map $\dot f_\theta$. The last line in (\ref{proof of modified delta method for the bootstrap - 60}) converges to $0$ as $n\rightarrow\infty$ by (\ref{modified delta method for the bootstrap - assumption - 18}) in assumption (f'). This gives (\ref{proof of modified delta method for the bootstrap - 30}), and the proof is complete.
\proofendsign


\appendix

\section{Weak topology and weak convergence for the open-ball $\sigma$-algebra}\label{sec weak conv for open ball}

Let $(\bE,d)$ be a metric space and ${\cal B}^\circ$ be the $\sigma$-algebra on $\bE$ generated by the open balls $B_r(x):=\{y\in\bE:d(x,y)<r\}$, $x\in\bE$, $r>0$. We will refer to ${\cal B}^\circ$ as {\em open-ball $\sigma$-algebra}. If $(\bE,d)$ is separable, then ${\cal B}^\circ$ coincides with the Borel $\sigma$-algebra ${\cal B}$. If $(\bE,d)$ is not separable, then ${\cal B}^\circ$ might be strictly smaller than ${\cal B}$ and thus a continuous real-valued function on $\bE$ is not necessarily $({\cal B}^\circ,{\cal B}(\R))$-measurable. Let $C_{\rm b}^\circ$ be the set of all bounded, continuous and $({\cal B}^\circ,{\cal B}(\R))$-measurable real-valued functions on $\bE$, and ${\cal M}_1^\circ$ be the set of all probability measures on $(\bE,{\cal B}^\circ)$. For every $f\in C_{\rm b}^\circ$ we consider the mapping
$$
    \pi_f:{\cal M}_1^\circ\longrightarrow\R,\qquad\mu\longmapsto\int f\,d\mu.
$$
The {\em weak$^\circ$ topology} ${\cal O}_{\rm w}^\circ$ on ${\cal M}_1^\circ$ is defined to be the topology ${\cal O}(\F)$ generated by the class of functions $\F:=\{\pi_f:f\in C_{\rm b}^\circ\}$. That is, ${\cal O}_{\rm w}^\circ:={\cal O}({\cal S}_\F):=\bigcap_{\mbox{\scriptsize{${\cal O}$ topology on ${\cal M}_1^\circ$ with ${\cal O}$}}\supseteq{\cal S}_\F}{\cal O}$ for the system ${\cal S}_\F:=\{\pi_f^{-1}(G'):f\in C_{\rm b}^\circ,\,G'\in{\cal O}_\R\}$, where ${\cal O}_\R$ is the usual topology of open sets in $\R$. In other words, the weak$^\circ$ topology is the coarsest topology on ${\cal M}_1^\circ$ w.r.t.\ which each of the maps $\pi_f$, $f\in C_{\rm b}^\circ$, is continuous. A sequence $(\mu_n)$ in ${\cal M}_1^\circ$ converges to some $\mu_0\in{\cal M}_1^\circ$ in the weak$^\circ$ topology ${\cal O}_{\rm w}^\circ$ if and only if
$$
    \int f\,d\mu_n\,\longrightarrow\,\int f\,d\mu_0\qquad\mbox{for all }f\in C_{\rm b}^\circ;
$$
see, for instance, Lemma 2.52 in \cite{AliprantisBorder1999} (take into account that every sequence is a net). In this case, we also say that $(\mu_n)$ {\em converges weak$^\circ$ly} to $\mu_0$ and write $\mu_n\Rightarrow^\circ\mu_0$. It is worth mentioning that two probability measures $\mu_0,\nu_0\in{\cal M}_1^\circ$ coincide if $\mu_0[\bE_0]=\nu_0[\bE_0]=1$ for some separable $\bE_0\in{\cal B}^\circ$ and $\int f\,d\mu_0=\int f\,d\nu_0$ for all uniformly continuous $f\in C_{\rm b}^\circ$; see, for instance, \citet[Theorem 6.2]{Billingsley1999}.

\begin{remarknorm}\label{remark appendix B coincide if separable} Recall that ${\cal B}^\circ={\cal B}$ when $(\bE,d)$ is separable. In this case we suppress the superscript $^\circ$ and write simply $C_{\rm b}$, ${\cal M}_1$, weak, ${\cal O}_{\rm w}$, and $\Rightarrow$ instead of $C_{\rm b}^\circ$, ${\cal M}_1^\circ$, weak$^\circ$, ${\cal O}_{\rm w}^\circ$, and $\Rightarrow^\circ$, respectively.
{\hspace*{\fill}$\Diamond$\par\bigskip}
\end{remarknorm}

Denote by ${\rm BL}_1^\circ$ the set of all $({\cal B}^\circ,{\cal B}(\R))$-measurable functions $f:\bE\rightarrow\R$ satisfying $|f(x)-f(y)|\le d(x,y)$ for all $x,y\in\bE$ and $\sup_{x\in\bE}|f(x)|\le 1$. Note that ${\rm BL}_1^\circ$ is contained in the set of all uniformly continuous functions in $C_{\rm b}^\circ$. The {\em bounded Lipschitz distance} on ${\cal M}_1^\circ$ is defined by
\begin{equation}\label{def bl metric - add on}
    d_{\scriptsize{\rm BL}}^\circ(\mu,\nu)\,:=\,\sup_{f\in{\rm BL}_1^\circ}\Big|\int f\,d\mu-\int f\,d\nu\Big|\,.
\end{equation}
It is easily seen that the mapping $d_{\scriptsize{\rm BL}}^\circ:{\cal M}_1^\circ\times{\cal M}_1^\circ\to\R_+$ satisfies the axioms of a pseudo-metric on ${\cal M}_1^\circ$, i.e.\ that it is symmetric and satisfies $d_{\scriptsize{\rm BL}}^\circ(\mu,\mu)=0$ as well as the triangle inequality.

\begin{remarknorm}\label{remark appendix b separable implies metric}
If $(\bE,d)$ is separable, then we again suppress the superscript $^\circ$ and write simply ${\rm BL}_1$ and $d_{\scriptsize{\rm BL}}$ instead of ${\rm BL}_1^\circ$ and $d_{\scriptsize{\rm BL}}^\circ$, respectively. In this case the bounded Lipschitz distance $d_{\scriptsize{\rm BL}}$ provides even a metric on ${\cal M}_1$, because ${\rm BL}_1$ is separating in ${\cal M}_1$; the latter follows from the proof of Theorem 1.2 in \cite{Billingsley1999}.
{\hspace*{\fill}$\Diamond$\par\bigskip}
\end{remarknorm}

\begin{theorem}\label{Portemanteau}
{\bf (Portmanteau theorem)} Let $\mu_n\in{\cal M}_1^\circ$, $n\in\N_0$, and assume that $\mu_0[\bE_0]=1$ for some separable $\bE_0\in{\cal B}^\circ$. Then the following conditions are equivalent:
\begin{itemize}
    \item[(a)] $\mu_n\Rightarrow^\circ\mu_0$.
    \item[(b)] $\int f\,d\mu_n\rightarrow\int f\,d\mu_0$ for all uniformly continuous $f\in C_{\rm b}^\circ$.
    \item[(c)] $\limsup_{n\to\infty}\mu_n[F]\le\mu_0[F]$ for all closed $F\in{\cal B}^\circ$.
    \item[(d)] $\liminf_{n\to\infty}\mu_n[G]\ge\mu_0[G]$ for all open $G\in{\cal B}^\circ$.
    \item[(e)] $\mu_n[A]\rightarrow\mu_0[A]$ for every $A\in{\cal B}^\circ$ for which ${\cal B}^\circ$ contains an open set $G$ and a closed set $F$ such that $G\subseteq A\subseteq F$ and $\mu_0[F\setminus G]=0$.
    \item[(f)] $\int f\,d\mu_n\rightarrow\int f\,d\mu_0$ for all $f\in {\rm BL}_1^\circ$.
    \item[(g)] $d_{\scriptsize{\rm BL}}^\circ(\mu_n,\mu_0)\rightarrow0$.
\end{itemize}
\end{theorem}

\begin{proof}
The equivalence of the conditions (a), (b), (c), (d), and (e) is known from Theorem 6.3 of \cite{Billingsley1999}, and the implications (b)$\Rightarrow$(f) is trivial. The arguments in the proof of (b)$\Rightarrow$(c) in Theorem 6.3 of \cite{Billingsley1999} also prove the implication (f)$\Rightarrow$(c). Indeed, the function $f$ defined in (6.1) in \cite{Billingsley1999} is bounded by $1$ and Lipschitz continuous with Lipschitz constant $\varepsilon^{-1}$, $\varepsilon f$ is an element of ${\rm BL}_1^\circ$ for $\varepsilon\in(0,1]$, and $\int f\,d\mu_n\rightarrow\int f\,d\mu$ if and only if $\int \varepsilon f\,d\mu_n\rightarrow\int\varepsilon f\,d\mu$. Finally, the equivalence of (a) and (g) was discussed in Example IV.3.22 of \citet{Pollard1984}.
\end{proof}

The following Theorem \ref{aliprantis} is a special case of Theorem 15.12 in \citet{AliprantisBorder1999}. Recall that a topological space is separable if it contains a countable dense subset; a subset is dense in a topological space if its closure coincides with the whole space.

\begin{theorem}\label{aliprantis}
The topological space $({\cal M}_1,{\cal O}_{\rm w})$ is metrizable and separable if $(\bE,d)$ is separable.
\end{theorem}

The bounded Lipschitz distance $d_{\scriptsize{\rm BL}}$ provides a metric on ${\cal M}_1$ when $(\bE,d)$ is separable; cf.\ Remark \ref{remark appendix b separable implies metric}.  Also recall that the topology generated by a metric consists of all $d$-open subsets of the underlying space. As a consequence of Theorem \ref{aliprantis} and the Portmanteau theorem \ref{Portemanteau} we obtain the following well known result.

\begin{corollary}
If $(\bE,d)$ is separable, then the bounded Lipschitz distance $d_{\scriptsize{\rm BL}}$ generates the weak topology ${\cal O}_{\rm w}$ on ${\cal M}_1$.
\end{corollary}

\begin{proof}
First, two topologies ${\cal O}$ and ${\cal O}'$ on a nonempty set coincide if and only if the identity is a homeomorphism w.r.t.\ ${\cal O}$ and ${\cal O}'$. Second, a topology is first countable if it is metrizable; cf.\ \citet[p.\,27]{AliprantisBorder1999}. Thus it follows by the second part of Theorem 2.40 in \cite{AliprantisBorder1999} that two metrizable topologies coincide if and only if convergence of any sequence in ${\cal O}$ implies convergence of the sequence in ${\cal O}'$ and vice versa. By Theorem \ref{aliprantis} the topology ${\cal O}_{\rm w}$ is metrizable, and the topology ${\cal O}(d_{\scriptsize{\rm BL}})$ generated by the metric $d_{\scriptsize{\rm BL}}$ is metrizable anyway. Thus the equivalence of (a) and (g) in Theorem \ref{Portemanteau} implies ${\cal O}_{\rm w}={\cal O}(d_{\scriptsize{\rm BL}})$, i.e.\ the metric $d_{\scriptsize{\rm BL}}$ indeed generates the weak topology ${\cal O}_{\rm w}$.
\end{proof}


\section{Convergence in distribution and convergence in probability for the open-ball $\sigma$-algebra}\label{appendix Conv Iin Dist and Prob}

Let $(\bE,d)$ be a metric space and ${\cal B}^\circ$ the open-ball $\sigma$-algebra on $\bE$. A sequence $(X_n)$ of $(\bE,{\cal B}^\circ)$-valued random variables is said to {\em converge in distribution$^\circ$} to an $(\bE,{\cal B}^\circ)$-valued random variable $X_0$ if the sequence $({\rm law}\{X_n\})$ weak$^\circ$ly converges to ${\rm law}\{X_0\}$. In this case, we write $X_n\leadsto^\circ X_0$. In the case where the random variables $X_n$, $n\in\N_0$, are all defined on the same probability space $(\Omega,{\cal F},\pr)$ the sequence $(X_n)$ is said to {\em converge in probability$^\circ$} to $X_0$ if the mappings $\omega\mapsto d(X_n(\omega),X_0(\omega))$, $n\in\N$, are $({\cal F},{\cal B}(\R_+))$-measurable and satisfy
\begin{equation}\label{def conv in prob - eq}
    \lim_{n\to\infty}\pr[d(X_n,X_0)\ge\varepsilon]=0\quad\mbox{ for all }\varepsilon>0.
\end{equation}
In this case, we write $X_n\rightarrow^{{\sf p},\circ} X_0$. As usual, by {\em $\pr$-almost sure convergence} of the sequence $(X_n)$ to $X_0$, abbreviated by $X_n\rightarrow X_0$ $\pr$-a.s., we will mean that there exists a set $N\in{\cal F}$ with that $\pr[N]=0$ and $d(X_n(\omega),X_0(\omega))\rightarrow 0$ for all $\omega\in\Omega\setminus N$.

\begin{proposition}\label{as conv implies conv prob}
Let $X_n$, $n\in\N_0$, be $(\bE,{\cal B}^\circ)$-valued random variables on a common probability space $(\Omega,{\cal F},\pr)$, and assume that the mappings $\omega\mapsto d(X_n(\omega),X_0(\omega))$, $n\in\N$, are $({\cal F},{\cal B}(\R_+))$-measurable. Then $X_n\rightarrow X_0$ $\pr$-a.s.\ implies $X_n\rightarrow^{{\sf p},\circ} X_0$.
\end{proposition}

\begin{proof}
By assumption the variables $d(X_n,X_0)$, $n\in\N$, are $({\cal F},{\cal B}(\R_+))$-measurable, and therefore the variable $\limsup_{n\to\infty}d(X_n,X_0)$ is $({\cal F},{\cal B}(\R_+))$-measurable. Since $X_n\rightarrow X_0$ $\pr$-a.s., we obtain $\pr[\limsup_{n\to\infty}d(X_n,X_0)=0]=1$. This implies
$$
    \pr\big[\limsup_{n\to\infty}\{d(X_n,X_0)\ge\varepsilon\}\big]\le\pr\big[\limsup_{n\to\infty}d(X_n,X_0)\ge\varepsilon\big]=0\quad\mbox{ for all }\varepsilon>0
$$
which together with the reverse of Fatou's lemma gives $\limsup_{n\to\infty}\pr[d(X_n,X_0)\ge\varepsilon]=0$ for every $\varepsilon>0$.
\end{proof}

When $X_0$ takes almost surely values in a separable measurable set, then convergence in probability$^\circ$ implies convergence in distribution$^\circ$ of $X_n$ to $X_0$:

\begin{proposition}\label{conv prob implies conv distr}
Let $X_n$, $n\in\N_0$, be $(\bE,{\cal B}^\circ)$-valued random variables on a common probability space $(\Omega,{\cal F},\pr)$, and assume that $\pr[X_0\in\bE_0]=1$ for some separable $\bE_0\in{\cal B}^\circ$. Then $X_n\rightarrow^{{\sf p},\circ} X_0$ implies $X_n\leadsto^\circ X_0$.
\end{proposition}

\begin{proof}
For any $f\in {\rm BL}_1^\circ$ we have $|\int f\,d\pr_{X_n}-\int f\,d\pr_{X_0}|\le2\,\pr[d(X_n,X_0)\ge\varepsilon/2]+\varepsilon/2$ for all $\varepsilon>0$, i.e.\ $\int f\,d\pr_{X_n}\to \int f\,d\pr_{X_0}$. The claim then follows by the implication (f)$\Rightarrow$(a) in the Portmanteau theorem \ref{Portemanteau}.
\end{proof}

The following lemma implies that the measurability condition in the definition of convergence in probability$^\circ$ is automatically satisfied when $X_0$ is constant, i.e.\ when $X_0(\cdot)=x$ for some $x\in\bE$.

\begin{lemma}\label{Skorohod - Lemma - Lemma}
For every $x\in\bE$, the mapping $y\mapsto d(x,y)$ is continuous and $({\cal B}^\circ,{\cal B}(\R))$-measurable.
\end{lemma}

\begin{proof}
The continuity is obvious, and the $({\cal B}^\circ,{\cal B}(\R))$-measurability follows by
$$
    \{d(x,\cdot)<a\}=\{y\in\bE:\,d(x,y)<a\}=B_{a}(x)\in{\cal B}^\circ\quad\mbox{for every $a>0$}
$$
and $\{d(x,\cdot)<a\}=\emptyset\in{\cal B}^\circ$ for every $a\le 0$.
\end{proof}

For constant $X_0$ we also have that convergence in probability$^\circ$ of $X_n$ to $X_0$ is equivalent to convergence in distribution$^\circ$ of $X_n$ to $X_0$:

\begin{proposition}\label{conv prob equiv conv distr}
Let $X_n$, $n\in\N$, be $(\bE,{\cal B}^\circ)$-valued random variables on a common probability space $(\Omega,{\cal F},\pr)$, and $x_0\in\bE$ be a constant. Then:
\begin{itemize}
    \item[(i)] $X_n\rightarrow x_0$ $\pr$-a.s.\ implies $X_n\rightarrow^{{\sf p},\circ} x_0$.
    \item[(ii)] $X_n\rightarrow^{{\sf p},\circ} x_0$ if and only if $X_n\leadsto^\circ x_0$.
\end{itemize}
\end{proposition}

\begin{proof}
Part (i) follows from Proposition \ref{as conv implies conv prob} and Lemma \ref{Skorohod - Lemma - Lemma}. To prove part (ii), first assume $X_n\leadsto^\circ x_0$. Set $f(x):=\min\{d(x,x_0);1\}$, $x\in\bE$, and note that $f\in C_{\rm b}^\circ$. By Markov's inequality and Lemma \ref{Skorohod - Lemma - Lemma} we obtain
$$
    \pr[d(X_n,x_0)\ge\varepsilon]\,\le\,\frac{1}{\varepsilon}\int f(X_n(\omega))\,\pr[d\omega]\,\longrightarrow\,\frac{1}{\varepsilon}\int f(x_0)\,\pr[d\omega]=\,0,\quad n\to\infty
$$
for every $\varepsilon>0$. That is, $X_n\rightarrow^\circ x_0$. The other direction in part (ii) follows from Proposition \ref{conv prob implies conv distr}, because the set $\{x_0\}=\bigcap_{n\in\N}B_{1/n}(x_0)$ is separable and lies in ${\cal B}^\circ$.
\end{proof}

Recall that ${\cal B}^\circ={\cal B}$ when $(\bE,d)$ is separable. In this case we suppress the superscript $^\circ$ and write simply $\leadsto$, $\rightarrow^{\sf p}$, convergence in distribution, and convergence in probability instead of $\leadsto^\circ$, $\rightarrow^{{\sf p},\circ}$, convergence in distribution$^\circ$, and convergence in probability$^\circ$, respectively.


\section{An extended Continuous Mapping theorem and a delta-method for the open-ball $\sigma$-algebra}\label{appendix Functional Delta-Method}

As mentioned in the introduction, Theorem \ref{modified delta method for the bootstrap} is based on a generalization of Theorem 4.1 in \citet{BeutnerZaehle2010}, which in turn is a generalization of the classical functional delta-method in the form of Theorem 3 of \cite{Gill1989}. The proof of the generalization of Theorem 4.1 in \citet{BeutnerZaehle2010} is based on the extended Continuous Mapping theorem \ref{Skorohod - CMT} below. An extended Continuous Mapping theorem for convergence in distribution for the Borel $\sigma$-algebra can be found in \citet[Theorem 4.27]{Kallenberg2002}. A corresponding result for convergence in distribution in the Hoffmann-J{\o}rgensen is given, for example, in \citet[Theorem 1.11.1]{van der Vaart Wellner 1996}. However, we could not find a version of this result for convergence in distribution$^\circ$ for the open-ball $\sigma$-algebra. So we include a proof for Theorem \ref{Skorohod - CMT}. Note that Theorem \ref{Skorohod - CMT} is a generalization of the ``ordinary'' Continuous Mapping theorem for convergence in distribution$^\circ$ for the open-ball $\sigma$-algebra as given by \citet[Theorem 6.4]{Billingsley1999}. Let $(\bE,d)$ and $(\widetilde\bE,d_{\widetilde\bE})$ be metric spaces and ${\cal B}^\circ$ and $\widetilde{\cal B}^\circ$ be the open-ball $\sigma$-algebras on $\bE$ and $\widetilde\bE$, respectively.

\begin{theorem}\label{Skorohod - CMT}
{\bf (Extended CMT for random variables)}
Let $\bE_n\subseteq\bE$ and $\xi_n$ be an $(\bE,{\cal B}^\circ)$-valued random variable on some probability space $(\Omega_n,{\cal F}_n,\pr_n)$ such that $\xi_n(\Omega_n)\subseteq\bE_n$, $n\in\N$. Let $\xi_0$ be an $(\bE,{\cal B}^\circ)$-valued random variable on some probability space $(\Omega_0,{\cal F}_0,\pr_0)$ such that $\xi_0(\Omega_0)\subseteq\bE_0$ for some separable $\bE_0\in{\cal B}^\circ$. Let $h_n: \bE_n \rightarrow\widetilde\bE$ be a map such that the map $h_n(\xi_n):\Omega_n\to\widetilde\bE$ is $({\cal F}_n,\widetilde{\cal B}^\circ)$-measurable, $n\in\N$. Let $h_0: \bE_0 \rightarrow\widetilde\bE$ be a $({\cal B}_0^\circ,\widetilde{\cal B}^\circ)$-measurable map, where ${\cal B}_0^\circ:={\cal B}^\circ\cap\bE_0$ ($\subseteq {\cal B}^\circ$). Moreover, assume that the following two assertions hold:
\begin{itemize}
    \item[(a)] $\xi_n \leadsto^\circ\xi_0$.
    \item[(b)] For every $x_n\in\bE_n$, $n\in\N_0$, we have $\widetilde d(h_n(x_n),h_0(x_0))\rightarrow 0$ when $d(x_n,x_0)\rightarrow 0$.
\end{itemize}
Then $h_n(\xi_n) \leadsto^\circ h_0(\xi_0)$.
\end{theorem}

\begin{remarknorm}
Note that we do {\em not} assume in Theorem \ref{Skorohod - CMT} that the maps $h_n$, $n\in\N$, are $({\cal B}^\circ,\widetilde{\cal B}^\circ)$-measurable. This implies that for $n\in\N$ the law $\pr_n\circ(h_n(\xi_n))^{-1}$ of $h_n(\xi_n)$ can not necessarily be represented as the image law of $\xi_n$'s law $\pr_n\circ\xi_n^{-1}$ w.r.t.\ $h_n$.
{\hspace*{\fill}$\Diamond$\par\bigskip}
\end{remarknorm}

\begin{proof}{\bf of Theorem \ref{Skorohod - CMT}}\,
According to the implication (d)$\Rightarrow$(a) in the Portmanteau theorem \ref{Portemanteau}, it suffices to show that $\liminf_{n\to\infty}\pr_n\circ h_n(\xi_n)^{-1}[\widetilde G]\ge\pr_0\circ h_0(\xi_0)^{-1}[\widetilde G]$ for every open set $\widetilde G\in\widetilde{\cal B}^\circ$. So, let $\widetilde G\in\widetilde{\cal B}^\circ$ be open. First we note that
\begin{equation}\label{Skorohod - CMT - Eric - PROOF - 10}
    h_0^{-1}(\widetilde G)\cap\bE_0\,\subseteq\,\bigcup_{m=1}^\infty\Big(\Big\{\bigcap_{k=m}^\infty h_k^{-1}(\widetilde G)\Big\}^{\scriptsize{\rm int}}\bigcap\,\bE_0\Big),
\end{equation}
where the superscript $^{\scriptsize{\rm int}}$ refers to the interior of a set. Indeed: For every $x_0\in h_0^{-1}(\widetilde G)\cap\bE_0$ there exists an $m\in\N$ and a neighborhood $U$ of $x_0$ such that $h_k(x)\in\widetilde G$ for all $k\ge m$ and $x\in U$. Otherwise we could find for every $m\in\N$ some $k_m\ge m$ and $x_m\in B_{1/m}(x_0)$ such that $h_{k_m}(x_m)\not\in\widetilde G$. But then we had $d(x_m,x_0)\rightarrow 0$ and $\widetilde d(h_{k_m}(x_m),h_0(x_0))\not\rightarrow 0$ (take into account that $h_0(x_0)\in\widetilde G$ and $\widetilde G$ is open), which contradicts assumption (b). Hence $U\subseteq\bigcap_{k=m}^\infty h_k^{-1}(\widetilde G)$ and thus $x_0\in\{\bigcap_{k=m}^\infty h_k^{-1}(\widetilde G)\}^{\scriptsize{\rm int}}$. In particular, $h_0^{-1}(\widetilde G)\cap\bE_0\subseteq\bigcup_{m=1}^\infty\{\bigcap_{k=m}^\infty h_k^{-1}(\widetilde G)\}^{\scriptsize{\rm int}}$. Now (\ref{Skorohod - CMT - Eric - PROOF - 10}) is obvious.

Further, for every $m\in\N$ we can find a union $G_m$ of countably many open balls such that
\begin{equation}\label{Skorohod - CMT - Eric - PROOF - 20}
    \Big\{\bigcap_{k=m}^\infty h_k^{-1}(\widetilde G)\Big\}^{\scriptsize{\rm int}}\bigcap\,\bE_0\,\subseteq\,G_m\,\subseteq\,\Big\{\bigcap_{k=m}^\infty h_k^{-1}(\widetilde G)\Big\}^{\scriptsize{\rm int}},
\end{equation}
and we may and do assume $G_1\subseteq G_2\subseteq\cdots$. To prove this one can proceed by an induction on $m$. First let $m=1$. For every $x\in\{\bigcap_{k=1}^\infty h_k^{-1}(\widetilde G)\}^{\scriptsize{\rm int}}$ we can find an open ball $B_{r_x}(x)$ around $x$ which is contained in $\{\bigcap_{k=1}^\infty h_k^{-1}(\widetilde G)\}^{\scriptsize{\rm int}}$, because the latter set is open. The system which consists of the open balls $B_{r_x}(x)$, $x\in \{\bigcap_{k=1}^\infty h_k^{-1}(\widetilde G)\}^{\scriptsize{\rm int}}$, provides an open cover of $\{\bigcap_{k=1}^\infty h_k^{-1}(\widetilde G)\}^{\scriptsize{\rm int}}\bigcap\,\bE_0$. Since the latter set is separable (recall that $\bE_0$ was assumed to be separable), Lindelöf's theorem ensures that there is a countable subcover. The set $G_1$ can now be defined as the union of the elements of this subcover. Next assume that $G_1,\ldots,G_M$ are unions of countably many open balls such that $G_1\subseteq\cdots\subseteq G_M$ and (\ref{Skorohod - CMT - Eric - PROOF - 20}) holds for $m=1,\ldots,M$. For every $x\in\{\bigcap_{k=M+1}^\infty h_k^{-1}(\widetilde G)\}^{\scriptsize{\rm int}}$ we can find an open ball $B_{r_x}(x)$ around $x$ which is contained in $\{\bigcap_{k=M+1}^\infty h_k^{-1}(\widetilde G)\}^{\scriptsize{\rm int}}$, because the latter set is open. The system which consists of $B_{r_x}(x)$, $x\in\{\bigcap_{k=M+1}^\infty h_k^{-1}(\widetilde G)\}^{\scriptsize{\rm int}}\setminus\{\bigcap_{k=M}^\infty h_k^{-1}(\widetilde G)\}^{\scriptsize{\rm int}}$ and of the countably many open balls which unify to $G_M$ provides an open cover of $\{\bigcap_{k=M+1}^\infty h_k^{-1}(\widetilde G)\}^{\scriptsize{\rm int}}\bigcap\,\bE_0$. Since the latter set is separable, Lindelöf's theorem ensures that there is a countable subcover. Without loss of generality we may and do assume that the countably many open balls which unify to $G_M$ belong to this countable subcover. Defining $G_{M+1}$ as the union of the elements of this subcover we obtain $G_M\subseteq G_{M+1}$ and (\ref{Skorohod - CMT - Eric - PROOF - 20}) for $m=M+1$.

As countable unions of open balls the sets $G_m$, $m\in\N$, are open and lie in ${\cal B}^\circ$. Then, using (\ref{Skorohod - CMT - Eric - PROOF - 10}), the first ``$\subseteq$'' in (\ref{Skorohod - CMT - Eric - PROOF - 20}), and the inclusions $G_1\subseteq G_2\subseteq\cdots$ (along with the continuity from below of $\pr_0\circ\xi_0^{-1}$),
\begin{eqnarray}\label{Skorohod - CMT - Eric - PROOF - 30}
    \pr_0\circ h_0(\xi_0)^{-1}\big[\widetilde G\big]
    & = & \pr_0\circ\xi_0^{-1}\big[h_0^{-1}(\widetilde G)\big]\nonumber\\
    & = & \pr_0\circ\xi_0^{-1}\big[h_0^{-1}(\widetilde G)\cap\bE_0\big]\nonumber\\
    & \le & \pr_0^{\scriptsize{\sf out}}\Big[\xi_0\in \bigcup_{m=1}^\infty\Big(\Big\{\bigcap_{k=m}^\infty h_k^{-1}(\widetilde G)\Big\}^{\scriptsize{\rm int}}\bigcap\,\bE_0\Big)\Big]\nonumber\\
    & \le & \pr_0\Big[\xi_0\in\bigcup_{m=1}^\infty G_m\Big]\nonumber\\
    & \le & \pr_0\circ\xi_0^{-1}\Big[\bigcup_{m=1}^\infty G_m\Big]\nonumber\\
    & = & \sup_{m\in\N}\,\pr_0\circ\xi_0^{-1}[G_m]\nonumber\\
    & \le & \sup_{m\in\N}\,\liminf_{n\to\infty}\,\pr_n\circ\xi_n^{-1}[G_m],
\end{eqnarray}
where the last step follows from assumption (a) and the implication (a)$\Rightarrow$(d) in the Portmanteau theorem \ref{Portemanteau}. Now, (\ref{Skorohod - CMT - Eric - PROOF - 30}) and the second ``$\subseteq$'' in (\ref{Skorohod - CMT - Eric - PROOF - 20}) yield
\begin{eqnarray*}
    \pr_0\circ h_0(\xi_0)^{-1}\big[\widetilde G\big]
    & \le & \sup_{m\in\N}\,\liminf_{n\to\infty}\,\pr_n^{\scriptsize{\sf out}}\Big[\xi_n\in\bigcap_{k=m}^\infty h_k^{-1}(\widetilde G)\Big]\\
    & \le & \liminf_{n\to\infty}\,\pr_n\big[\xi_n\in h_n^{-1}(\widetilde G)\big] \\
    & = & \liminf_{n\to\infty}\,\pr_n\circ h_n(\xi_n)^{-1}\big[\widetilde G\big].
\end{eqnarray*}
This completes the proof.
\end{proof}

Before giving the generalization of Theorem 4.1 in \citet{BeutnerZaehle2010} we recall the definition of quasi-Hadamard differentiability. For this let $\V$ and $\widetilde\bE$ be vector spaces, and $\bE\subseteq\V$ be a subspace of $\V$. Let $\|\cdot\|_{\bE}$ and $\|\cdot\|_{\widetilde \bE}$ be norms on $\bE$ and $\widetilde\bE$, respectively.

\begin{definition}\label{definition quasi hadamard}
{\bf (Quasi-Hadamard differentiability)}
Let $H:\V_H\rightarrow\widetilde\bE$ be a map defined on some $\V_H\subseteq\V$, and $\bE_0$ be a subset of $\bE$. Then $H$ is said to be quasi-Hadamard differentiable at $x \in \V_H$ tangentially to $\bE_0\langle\bE\rangle$ if there is some continuous map $\dot H_{x}:\bE_0\rightarrow\widetilde\bE$ such that
\begin{eqnarray}\label{def eq for HD}
    \lim_{n\to\infty}\Big\| \dot H_x(x_0)-\frac{H(x+\varepsilon_nx_n)-H(x)}{\varepsilon_n}\Big\|_{\widetilde\bE}\,=\,0
\end{eqnarray}
holds for each triplet $(x_0,(x_n),(\varepsilon_n))$, with $x_0\in\bE_0$, $(x_n)\subseteq\bE$ satisfying $\|x_n-x_0\|_{\bE}\to 0$ as well as $(x+\varepsilon_nx_n)\subseteq\V_H$, and $(\varepsilon_n)\subset(0,\infty)$ satisfying $\varepsilon_n\to 0$. In this case the map $\dot H_x$ is called quasi-Hadamard derivative of $H$ at $x$ tangentially to $\bE_0\langle\bE\rangle$.
\end{definition}

Recall that $\widetilde\bE$ is a vector space equipped with a norm $\|\cdot\|_{\widetilde\bE}$, and let $0_{\widetilde\bE}$ denote the null in $\widetilde\bE$. Set $\overline{\widetilde\bE}:=\widetilde\bE\times\widetilde\bE$ and let $\overline{\widetilde{\cal B}^\circ}$ be the $\sigma$-algebra on $\overline{\widetilde\bE}$ generated by the open balls w.r.t.\ the metric $\overline{\widetilde d}((\widetilde x_1,\widetilde x_2),(\widetilde y_1,\widetilde y_2)):=\max\{\|\widetilde x_1-\widetilde y_1\|_{\widetilde\bE};\|\widetilde x_2-\widetilde y_2\|_{\widetilde\bE}\}$. Recall that $\overline{\widetilde{\cal B}^\circ}\subseteq\widetilde{\cal B}^\circ\otimes\widetilde{\cal B}^\circ$, because any $\overline{\widetilde{d}}$-open ball in $\overline{\widetilde{\bE}}$ is the product of two $\|\cdot\|_{\widetilde\bE}$-open balls in $\widetilde\bE$. Let $(\Omega_n,{\cal F}_n,\pr_n)$ be a probability space and $X_n:\Omega_n\rightarrow\bE$ be any map, $n\in\N$. Recall that $\leadsto^\circ$ and $\rightarrow^\circ$ refer to convergence in distribution$^\circ$ and convergence in probability$^{{\sf p},\circ}$, respectively.

\begin{theorem}\label{modified delta method}
{\bf (Delta-method)} Let $H:\V_H\to\widetilde\bE$ be a map defined on some $\V_H\subseteq\bE$, and $x\in\V_H$. Let $\bE_0\in{\cal B}^\circ$ be some $\|\cdot\|_{\bE}$-separable subset of $\bE$. Let $(a_n)$ be a sequence of positive real numbers tending to $\infty$, 
and consider the following conditions:
\begin{itemize}
    \item[(a)] $X_n$ takes values only in $\V_H$.
    \item[(b)] $a_n(X_n-x)$ takes values only in $\bE$, is $({\cal F}_n,{\cal B}^\circ)$-measurable and satisfies
    \begin{eqnarray}\label{modified delta method - assumption}
         a_n(X_n-x) \leadsto^\circ X_0\qquad\mbox{in $(\bE,{\cal B}^\circ,\|\cdot\|_{\bE})$}
     \end{eqnarray}
    for some $(\bE,{\cal B}^\circ)$-valued random variable $X_0$ on some probability space $(\Omega_0,{\cal F}_0,\pr_0)$ with $X_0(\Omega_0)\subseteq\bE_0$.
    \item[(c)] $a_n(H(X_n)-H(x))$ is $({\cal F}_n,\widetilde{\cal B}^\circ)$-measurable.
    \item[(d)] The map $H$ is quasi-Hadamard differentiable at $x$ tangentially to $\bE_0\langle\bE\rangle$ with quasi-Hadamard derivative $\dot H_x:\bE_0\rightarrow\widetilde\bE$.
    \item[(e)] $(\Omega_n,{\cal F}_n,\pr_n)=(\Omega,{\cal F},\pr)$ for all $n\in\N$.
    \item[(f)] The quasi-Hadamard derivative $\dot H_x$ can be extended to $\bE$ such that the extension $\dot H_x:\bE\rightarrow\widetilde\bE$ is continuous at every point of $\bE_0$ and $({\cal B}^\circ,\widetilde{\cal B}^\circ)$-measurable.
    \item[(g)] The map $h:\overline{\widetilde\bE}\rightarrow\widetilde\bE$ defined by $h(\widetilde x_1,\widetilde x_2):=\widetilde x_1-\widetilde x_2$ is $(\overline{\widetilde{\cal B}^\circ},\widetilde{\cal B}^\circ)$-measurable.
\end{itemize}
Then the following two assertions hold:
\begin{itemize}
    \item[(i)] If conditions (a)--(d) hold true, then $\dot H_x(X_0)$ is $({\cal F}_0,\widetilde{\cal B}^\circ)$-measurable and
    $$
        a_n\big(H(X_n)-H(x)\big)\,\leadsto^\circ\,\dot H_x(X_0)\qquad\mbox{in $(\widetilde\bE,\widetilde{\cal B}^\circ,\|\cdot\|_{\widetilde\bE})$}.
    $$
    \item[(ii)] If conditions (a)--(g) hold true, then
        \begin{equation}\label{modified fct delta eq - ii}
            a_n\big(H(X_n)-H(x)\big)-\dot H_x\big(a_n(X_n-x)\big)\,\rightarrow^{{\sf p},\circ}\,0_{\widetilde\bE}\qquad\mbox{in $(\widetilde\bE,\|\cdot\|_{\widetilde\bE})$}.
    \end{equation}
\end{itemize}
\end{theorem}

\begin{remarknorm}
It is apparent from the following proof that for part (i) of Theorem \ref{modified delta method} it is not necessary to assume (as in Definition \ref{definition quasi hadamard}) that the quasi-Hadamard derivative $\dot H_x$ is continuous. It would suffice to require in Definition \ref{definition quasi hadamard} that the map $\dot H_x$ is $({\cal B}_0^\circ,\widetilde{{\cal B}}^\circ)$-measurable for the trace $\sigma$-algebra ${\cal B}_0^\circ:={\cal B}^\circ\cap\bE_0$ ($\subseteq{\cal B}^\circ$).
{\hspace*{\fill}$\Diamond$\par\bigskip}
\end{remarknorm}

\begin{proof}{\bf of Theorem \ref{modified delta method}}\,
For the proof of part (i) we adapt the arguments in the proof of Theorem 3.9.4 in \cite{van der Vaart Wellner 1996}, which then allow for an easy proof of part (ii).

(i): For every $n\in\N$, let $\bE_{n}:=\{x_n\in\bE:a_n^{-1}x_n+x \in \V_H\}$ and define the map $h_n:\bE_{n} \rightarrow \widetilde\bE$ by
$$
    h_n(x_n)\,:=\,\frac{H(x+a_n^{-1}x_n)-H(x)}{a_n^{-1}}\,.
$$
Moreover, define the map $h_0: \bE_0 \rightarrow \widetilde\bE$ by
$$
    h_0(x_0)\,:=\,\dot H_x(x_0).
$$
Now, the claim would follow by the extended Continuous Mapping theorem \ref{Skorohod - CMT} applied to the functions $h_n$, $n\in\N_0$, and the random variables $\xi_n:=a_n(X_n-x)$, $n\in\N$, and $\xi_0:=X_0$ if we can show that the assumptions of Theorem \ref{Skorohod - CMT} are satisfied. First, $\xi_n(\Omega_n)\subseteq\bE_n$ and $\xi_0(\Omega_0)\subseteq\bE_0$ clearly hold. Second, by assumption (c) we have that $h_n(\xi_n)=a_n(H(X_n)-H(x))$ is $({\cal F}_n,\widetilde{\cal B}^\circ)$-measurable.
Third, the map $h_0$ is continuous by assumption (on the quasi-Hadamard derivative). Thus $h_0$ is $({\cal B}_0^\circ,\widetilde{\cal B}^\circ)$-measurable, because the trace $\sigma$-algebra ${\cal B}_0^\circ:={\cal B}^\circ\cap\bE_0$ coincides with the Borel $\sigma$-algebra on $\bE_0$ (recall that $\bE_0$ is separable). In particular, $\dot H_x(X_0)$ is $({\cal F}_0,\widetilde{\cal B}^\circ)$-measurable. Fourth, condition (a) of Theorem \ref{Skorohod - CMT} holds by assumption (b). Fifth, condition (b) of Theorem \ref{Skorohod - CMT} is ensured by assumption (d) (note that (d) implies (\ref{def eq for HD})).

(ii): For every $n\in\N$, let $\bE_n$ and $h_n$ be as above and define the map $\overline h_n: \bE_{n} \rightarrow \overline{\widetilde\bE}$ by
$$
    \overline h_n(x_n)\,:=\,(h_n(x_n),\dot H_x(x_n)).
$$
Moreover, define the map $\overline h_0: \bE_0 \rightarrow\overline{\widetilde\bE}$ by
$$
    \overline h_0(x_0)\,:=\,(h_0(x_0),\dot H_x(x_0))\,=\,(\dot H_x(x_0),\dot H_x(x_0)).
$$
We will first show that
\begin{equation}\label{modified delta method PROOF 10}
    \overline h_n(a_n(X_n-x))\,\leadsto^\circ\, \overline h_0(X_0)\qquad\mbox{in $(\overline{\widetilde{\bE}},\overline{\widetilde{\cal B}^\circ},\overline{\widetilde d})$}.
\end{equation}
For (\ref{modified delta method PROOF 10}) it suffices to show that the assumption of the extended Continuous Mapping theorem \ref{Skorohod - CMT} applied to the functions $\overline h_n$ and $\xi_n$ (as defined above) are satisfied. The claim then follows by Theorem \ref{Skorohod - CMT}. First, we have already observed that $\xi_n(\Omega_n)\subseteq\bE_n$ and $\xi_0(\Omega_0)\subseteq\bE_0$. Second, we have seen in the proof of part (i) that $h_n(\xi_n)$ is $({\cal F}_{n},\widetilde{\cal B}^\circ)$-measurable, $n \in \N$. By assumption (f) the extended map $\dot H_x:\bE\rightarrow\widetilde\bE$ is $({\cal B}^\circ,\widetilde{\cal B}^\circ)$-measurable, which implies that $\dot H_x(\xi_n)$ is $({\cal F}_{n},\widetilde{\cal B}^\circ)$-measurable. Thus, $\overline h_n(\xi_n)=(h_n(\xi_n),\dot H_x(\xi_n))$ is $({\cal F}_{n},\widetilde{\cal B}^\circ\otimes\widetilde{\cal B}^\circ)$-measurable (to see this note that, in view of $\widetilde{\cal B}^\circ\otimes\widetilde{\cal B}^\circ=\sigma(\pi_1,\pi_2)$ for the coordinate projections $\pi_1,\pi_2$ on $\overline{\widetilde E}={\widetilde E}\times{\widetilde E}$, Theorem 7.4 of \cite{Bauer2001} shows that the map $(h_n(\xi_n),\dot H_x(\xi_n))$ is $({\cal F}_{n},\widetilde{\cal B}^\circ\otimes\widetilde{\cal B}^\circ)$-measurable if and only if the maps $h_n(\xi_n)=\pi_1\circ(h_n(\xi_n),\dot H_x(\xi_n))$ and $\dot H_x(\xi_n)=\pi_2\circ(h_n(\xi_n),\dot H_x(\xi_n))$ are $({\cal F}_n,{\widetilde{\cal B}}^\circ)$-measurable). In particular, the map $\overline h_n(\xi_n)=(h_n(\xi_n),\dot H_x(\xi_n))$ is $({\cal F}_{n},\overline{\widetilde{\cal B}^\circ})$-measurable, $n\in\N$. Third, we have seen in the proof of part (i) that the map $h_0=\dot H_x$ is $({\cal B}_0^\circ,\widetilde{\cal B}^\circ)$-measurable. Thus the map $\overline h_0$ is $({\cal B}_0^\circ,\widetilde{\cal B}^\circ\otimes\widetilde{\cal B}^\circ)$-measurable (one can argue as above) and in particular $({\cal B}_0^\circ,\overline{\widetilde{\cal B}^\circ})$-measurable. Fourth, condition (a) of Theorem \ref{Skorohod - CMT} holds by assumption (b). Fifth, condition (b) of Theorem \ref{Skorohod - CMT} is ensured by assumption (d) and the continuity of the extended map $\dot H_x$ at every point of $\bE_0$ (recall assumption (f)). Hence, (\ref{modified delta method PROOF 10}) holds.

By assumption (g) and the ordinary Continuous Mapping theorem (cf.\ \citet[Theorem 6.4]{Billingsley1999}) applied to (\ref{modified delta method PROOF 10}) and the map $h:\overline{\widetilde\bE}\rightarrow\widetilde\bE$, $(\widetilde x_1,\widetilde x_2)\mapsto\widetilde x_1-\widetilde x_2$, we now have
$$
    h_n(a_n(X_n-x))-\dot H_x(a_n(X_n-x))\,\leadsto^\circ\,\dot H_x(X_0)-\dot H_x(X_0),
$$
i.e.
$$
    a_n\big(H(X_n)-H(x)\big)-\dot H_x\big(a_n(X_n-x)\big)\,\leadsto^\circ\,0_{\widetilde\bE}.
$$
By Proposition \ref{conv prob equiv conv distr} we can conclude (\ref{modified fct delta eq - ii}).
\end{proof}


\section{Probability kernels and conditional distributions}\label{section topology 30}

Let $(\Omega,{\cal F})$ be a measurable space. Let $(\bE,d)$ be a metric space and ${\cal B}^\circ$ be the open-ball $\sigma$-algebra on $\bE$. A map $P:\Omega\times{\cal B}^\circ\rightarrow[0,1]$ is said to be a {\em probability kernel} from $(\Omega,{\cal F})$ to $(\bE,{\cal B}^\circ)$ if $P(\,\cdot\,,A)$ is $({\cal F},{\cal B}([0,1]))$-measurable for every $A\in{\cal B}^\circ$, and $P(\omega,\,\cdot\,)$ is a probability measure on $(\bE,{\cal B}^\circ)$ for every $\omega\in\Omega$. Of course, we may regard $P$ as a map from $\Omega$ to ${\cal M}_1^\circ$. Recall that ${\cal M}_1^\circ={\cal M}_1$ when $(\bE,d)$ is separable. If in this case the set ${\cal M}_1$ is equipped with the weak topology ${\cal O}_{\rm w}$, then a probability kernel can be regarded as an ${\cal M}_1$-valued random variable (w.r.t.\ any probability measure on $(\Omega,{\cal F})$):

\begin{lemma}\label{topology 200}
Let $(\bE,d)$ be separable and $P$ be a probability kernel from $(\Omega,{\cal F})$ to $(\bE,{\cal B})$. Then the mapping $\omega\mapsto P(\omega,\bullet)$ is $({\cal F},\sigma({\cal O}_{\rm w}))$-measurable.
\end{lemma}

\begin{proof}
Since $(\bE,d)$ was assumed to be separable, the proof of the implication (4)$\Rightarrow$(1) in Theorem 19.7 in \citet{AliprantisBorder1999} shows that $\sigma({\cal O}_{\rm w})$ equals the $\sigma$-algebra generated by the system $\{\pi_f^{-1}(A): f \in C_{\rm b},\,A\subseteq\R\mbox{ open}\}$. So it suffices to show that the set
$$
    P(\,\cdot\,,\bullet)^{-1}(\pi_f^{-1}(A))\,=\,\pi_f(P(\,\cdot\,,\bullet))^{-1}(A)\,=\,\Big(\int f(x)\,P(\,\cdot\,,dx)\Big)^{-1}(A)
$$
is contained in ${\cal F}$ for every open $A \subseteq\R$ and $f\in C_{\rm b}$. But this follows from the well known fact (see e.g.\ Lemma 1.41 in \cite{Kallenberg2002}) that the mapping $\omega\mapsto\int f(x)P(\omega,dx)$ is $({\cal F},{\cal B}(\R))$-measurable for every $f\in C_{\rm b}$. This finishes the proof.
\end{proof}

Now, let $(\Omega',{\cal F}')$ and $(\bD,{\cal D})$ be further measurable spaces. Let $\pr$ and $\pr'$ be probability measures on respectively $\Omega$ and $\Omega'$, and set $(\overline\Omega,\overline{\cal F},\overline{\pr}):=(\Omega\times\Omega',{\cal F}\otimes{\cal F}',\pr\otimes\pr')$. Let $Y:\Omega\rightarrow\bD$ be an $({\cal F},{\cal D})$-measurable map and $X:\overline\Omega\rightarrow\bE$ be an $(\overline{\cal F},{\cal B}^\circ)$-measurable map. Note that $Y$ can also be regarded as a $(\bD,{\cal D})$-valued random variable on $(\overline\Omega,\overline{\cal F},\overline{\pr})$, and we are doing that in Lemma \ref{reresentation of P n}. The following lemma shows that under an additional assumption, the conditional distribution of $X$ given $Y$ can be specified explicitly.

\begin{lemma}\label{reresentation of P n}
Assume that $X(\omega,\omega')=g(Y(\omega),\omega')$ holds for all $(\omega,\omega')\in\overline\Omega$ and some $({\cal D}\otimes{\cal F}',{\cal B}^\circ)$-measurable map $g:\bD\times\Omega'\rightarrow\bE$. Then the map $P:\overline\Omega\times{\cal B}^\circ\rightarrow[0,1]$ defined by
$$
    P((\omega,\omega'),A)\,:=\,P(\omega,A)\,:=\,\pr'\circ X(\omega,\cdot)^{-1}[A],\qquad (\omega,\omega')\in\overline\Omega,\,A\in{\cal B}^\circ
$$
provides a conditional distribution of $X$ given $Y$.
\end{lemma}

\begin{proof}
First, $P$ provides a probability kernel from $(\overline{\Omega},\overline\sigma(Y))$ to $(\bE,{\cal B}^\circ)$. Indeed: The mapping $\widetilde\omega'\mapsto X(\omega,\widetilde\omega')$ is $({\cal F}',{\cal B}^\circ)$-measurable for every fixed $\omega\in\Omega$, because $X$ is $(\overline{\cal F},{\cal B}^\circ)$-measurable. So it immediately follows that the mapping $A'\mapsto P(\omega,A')$ is a probability measure on $(\bE,{\cal B}^\circ)$ for every $\omega\in\Omega$. Further, the mapping $(\omega,\widetilde\omega')\mapsto (Y(\omega),\widetilde\omega')$ is clearly $(\sigma(Y)\otimes{\cal F}',{\cal D}\otimes{\cal F}')$-measurable, which implies that the mapping $(\omega,\widetilde\omega')\mapsto X(\omega,\widetilde\omega')=g(Y(\omega),\widetilde\omega')$ is $(\sigma(Y)\otimes{\cal F}',{\cal B}^\circ)$-measurable. By Tonelli's part of Fubini's theorem it follows that the mapping $\omega\mapsto \int\eins_{A}(X(\omega,\widetilde\omega'))\,\pr'[d\widetilde\omega']=P(\omega,A)$ is $(\sigma(Y),{\cal B}([0,1]))$-measurable for every $A\in{\cal B}^\circ$. In particular, the mapping $(\omega,\omega')\mapsto P((\omega,\omega'),A)=P(\omega,A)$ is $(\overline\sigma(Y),{\cal B}([0,1]))$-measurable for every $A\in{\cal B}^\circ$.

Second, by Fubini's theorem we obtain for every $B\in{\cal D}$ and $A\in{\cal B}^\circ$,
\begin{eqnarray*}
    \int_{\{Y\in B\}} P((\omega,\omega'),A)\,\overline\pr[d(\omega,\omega')]
    & = & \int_{\{Y\in B\}} \pr'\circ X(\omega,\cdot)^{-1}[A]\,\overline\pr[d(\omega,\omega')]\\
    & = & \int_{\{Y\in B\}} \pr'\circ X(\omega,\cdot)^{-1}[A]\,\pr[d\omega]\\
    & = & \iint\eins_{\{Y\in B\}}(\omega)\,\eins_{\{X(\omega,\cdot)\in A\}}(\omega')\,\pr'[d\omega']\,\pr[d\omega]\\
    & = & \int\eins_{\{Y\in B\}}(\omega)\,\eins_{\{X(\omega,\cdot)\in A\}}(\omega')\,\overline\pr[d(\omega,\omega')]\\
    & = & \overline\pr\big[\{Y\in B\}\cap\{X\in A\}\big].
\end{eqnarray*}
This completes the proof.
\end{proof}


\section*{Acknowledgement}

The second author gratefully acknowledges support by BMBF through the project HYPERMATH under grant 05M13TSC.


\end{document}